\documentclass[11pt]{article}
\usepackage{graphicx,color}
\usepackage{amssymb,amsmath,amsfonts}
\usepackage{amsmath,mathrsfs,bm,url,times}
\usepackage{latexsym}
\usepackage{subfigure}
\usepackage{amsthm}
\newtheorem{proposition}{Proposition}
\newtheorem{theorem}{Theorem}
\newtheorem{lemma}{Lemma}
\newtheorem{corollary}{Corollary}
\theoremstyle{remark}
\newtheorem{definition}{Definition}

\newtheorem{remark}{Remark}

\newcommand{\R}{{\mathbb R}}

\newcommand{\onetom}{1,\dots,m}

\begin{document}
\title{Stability of phase difference trajectories of networks of Kuramoto oscillators with time-varying couplings and intrinsic frequencies}
\author{Wenlian Lu\thanks{School of Mathematical Sciences, Fudan University, Shanghai 200433, China (e-mail: wenlian@fudan.edu.cn).}\and Fatihcan M. Atay\thanks{Department of Mathematics, Bilkent University, 06800 Bilkent, Ankara, Turkey (e-mail: atay@member.ams.org).}}

\maketitle
\begin{abstract}
We study dynamics of phase-differences (PDs) of coupled oscillators where both the intrinsic frequencies and the couplings vary in time. In case the coupling coefficients are all nonnegative, we prove that  the PDs are asymptotically stable if there exists $T>0$ such that the aggregation of the time-varying graphs across any time interval of length $T$ has a spanning tree. We also consider the situation that the coupling coefficients may be negative and provide sufficient conditions for the asymptotic stability of the PD dynamics. Due to time-variations, the PDs are asymptotic to time-varying patterns rather than constant values. Hence, the PD dynamics can be regarded as a generalisation of the well-known phase-locking phenomena. We explicitly investigate several particular cases of time-varying graph structures, including asymptotically periodic PDs due to periodic coupling coefficients and intrinsic frequencies, small perturbations, and fast-switching near constant coupling and frequencies, which lead to PD dynamics close to a phase-locked one. Numerical examples are provided to illustrate the theoretical results.
\end{abstract}

\section{Introduction}

The Kuramoto model of coupled oscillators \cite{Kura1,Kura2} has been one of the most popular mathematical model to describe collective dynamics in, for instance, neural systems \cite{Break}, power grids \cite{Mach}, and seismology \cite{Vas}, due to its ability to describe the phase dynamics of coupled systems \cite{Kopell1,Kopell2}. A standard first-order Kuramoto model can be described as follows:
\begin{equation}
\dot{\theta}_{i}=\omega_{i}+\sum_{j=1}^{n}a_{ij}\sin(\theta_{j}-\theta_{i}),\quad i=\onetom,  \label{Kura}
\end{equation}
where $\theta_{i}\in S^1$ is the phase of the $i$-th oscillator, $\omega_{i}$ is its intrinsic frequency, and $a_{ij}$ is the coupling strength measuring the strength of the influence of oscillator $j$ on $i$.
Among the rich spectrum of dynamics \eqref{Kura} possesses, synchronization phenomenon has attracted a lot of interest from diverse fields.
Also known as phase-locked equilibrium, synchronization refers to the state where oscillators
lock their phase differences (PD) via local interactions, namely, the limit $\lim_{t\to\infty}(\theta_{i}(t)-\theta_{j}(t))$ exists for all $i,j$. This model exhibits phase transitions at critical values of
coupling, beyond which a collective behavior is achieved \cite{Peng}.

Meanwhile, the last two decades have witnessed the new field of network science bringing new insights into the study of models of collective behavior, such as the Kuramoto model \eqref{Kura}, where the set $\{a_{ij}\}$ in (\ref{Kura}) is identified with a (weighted) graph structure. New results have been obtained with the help of the emerging new methodologies, such as the dimension reduction ansatz \cite{Ott1,Ott2} and the consensus analysis in networked system \cite{Olf,Jad1} with the Lyapunov function method, and the effects of small-world and scale-free structures on synchronization were studied \cite{Gra,More1}. For more details, we refer to the comprehensive review literature \cite{Ace,Rod} and the references therein.

The majority of the existing literature is concerned with networks with static topology and couplings. However, many real-world applications from the social, natural, and engineering disciplines include a temporal variation in the topology of the network.
In communication networks, for example, some connections may fail due to occurrence of an obstacle between agents \cite{Olf2} and new connections may be created when one agent enters the effective region of other agents \cite{Vic,Jad}.
Time variability in the system structure has been experimentally reported for brain signals \cite{Rud,Shee}. Hence, there are important cases where the model should be formulated with time-varying parameters, which may lead nonequilibrium dynamics. {\color{blue} Synchronization of time-varying networks has recently attracted a lot of attention in the scientific literature \cite{Bely,Sku,Porfiri2006,Stil,Porfiri2007,LAJ1,LAJ2,Fuji,YLC2013,Levis}.} However, only very few papers have studied time-varying (also termed time-dependent) parameters in the Kuramoto model. In \cite{Pet,Pie}, the techniques of order parameters, thermodynamic limits, and the Ott-Antonsen ansatz as well as the dimension reduction method were extended to treat the Kuramoto model with a time-varying coupling that originates from
another nonconstant mean field \cite{Shee2}, and stable time-dependent collective dynamics were identified. The time-varying model was also investigated from a control point of view. In \cite{Lea}, minimising the $L_{2}$-norm of time-varying coupling were studied subject to synchrony performance, and in \cite{Fran}, input-to-state stability was considered. In addition, negative couplings should also be considered in a number of physical scenarios, for instance, in repressive synaptic couplings from inhibitory neurons in neural systems, inhibitory/inactive interactions in genetic regulation networks, and hostile relationship in social networks.

In this paper, we study phase dynamics of the Kuramoto model
where both coupling strengths and frequencies are time varying,
and additionally the coupling strengths are allowed to assume negative values.
Specifically, we consider the system
\begin{equation}
\dot{\theta}_{i}=\omega_{i}(t)+\sum_{j=1}^{m}a_{ij}(t)\sin(\theta_{j}-\theta_{i}),\quad i=\onetom,
\label{TVKR}
\end{equation}
where $\omega_{i}$ and $a_{ij}$ are varying with respect to time.
We mathematically formulate the non-equilibrium dynamics in the model by the phase differences (PDs) between oscillators. {\color{blue}In comparison, the existing literature mostly uses self-consistent solutions \cite{Ace} to investigate the phase difference between individual oscillators and the mean-field frequency \cite{Pet2013}, derive empirical criterions of stability for the distributions of parameters \cite{Ott1,Ott2,Iat}, and discuss the cases of phase shift in the coupling function \cite{Omel}.}

We derive and prove a series of sufficient conditions that guarantee that the PDs are asymptotically stable, in particular for the scenarios when negative couplings occur. In general, asymptotically stable PDs need not be constants but may be functions of time.
We study three specific scenarios of time-variability, which include periodicity, small perturbations, and fast-switching in the time-varying couplings and intrinsic frequencies. We identify the phase-unlocking dynamics in each case.

This paper is organized as follows. In Section 2, asymptotic stability of PDs is investigated. Particular cases of asymptotic PD dynamics are studied in Section 3 with numerical examples. Section 4 concludes the paper.

\textbf{Notation.} $\R^{n}$ and $\mathbb C^{n}$ stand for the $n$-dimensional Euclidean real and complex spaces, respectively. For a symmetric square matrix $B\in\R^{m,m}$, we order the eigenvalues as $\lambda_{1}(B)\le\lambda_{2}(B)\le\cdots\le\lambda_{m}(B)$, counting multiplicities. The Euclidean norm of a vector and the matrix induced by it is denoted by $\|\cdot\|$. For a subspace $\mathscr L$ in $\R^{m}$ and an $\mathscr L$-invariant matrix $U\in\R^{m,m}$ (i.e., $Uy\in\mathscr L$ for all $y\in\mathscr L$), the matrix norm $\|\cdot\|_{\mathscr L}$
is defined by
\begin{equation*}
\|U\|_{\mathscr L}=\max_{y\in\mathscr L,~\|y\|=1}\|Uy\|.
\end{equation*}
The set of nonnegative integers is denoted by $\mathbb{Z}^+$ and positive integers by $\mathbb{N}$.
Denote $[z]^{-}=\min\{z,0\}$. Boldface ${\mathbf 1}$ stands for the column vector of proper dimensions with all components equal to $1$, $o(\epsilon)$ denotes the infinitesimal as $\epsilon\to 0$, and $\lfloor z\rfloor$ stands for the largest integer less than or equal to $z$.

\section{Stability analysis}

Let $\mathcal G=\{V,E,A\}$ be a directed, weighted and signed graph, where $V=\{\onetom\}$ stands for the node set and $E$ for the link set, such that $(i,j)\in E$ if there is a link from node $j$ to node $i$, and $A=\{a_{ij}\}$ stands for the weight set. It holds that $(i,j)\in E$ if and only if $a_{ij}\ne 0$.
We do not consider self-links, i.e., $a_{ii}=0$ $\forall i$. The (signed) Laplacian of the graph is defined as $L=[l_{ij}]_{i,j=1}^{m}$ with $l_{ij}=-a_{ij}$ for $i\ne j$ and $l_{ii}=-\sum_{j\ne i}l_{ij}$.
In the particular case of nonnegative coupling coefficients, $l_{ij}\le 0$ for all $i\ne j$ so that $-L$ is a Metzler matrix.
We use $N_{i}=\{j:(i,j)\in E\}$ to denote the (in-)neighborhood of node $i$. The set of nodes having links to both $i$ and $j$ is denoted by $\Lambda_{ij}=\{k: a_{ik}>0~{\rm and}~a_{jk}>0\}$.
 For $\eta>0$, the threshold graph (or, the $\eta$-graph) of $L$ is defined as that the graph whose link set $E$ is composed of those edges $(i,j)$ with $l_{ij}<-\eta$.

The notation extends to the time-varying case in an obvious way. Thus, the time-varying adjacency matrix $A(t)=[a_{ij}(t)]$ corresponds to a dynamical graph $\mathcal G(t)$ with a (fixed) node set $V$ and time-varying link set $E(t)$. The time-varying Laplacian $L(t)$ has components $l_{ij}(t)=-a_{ij}(t)$ for $i\ne j$ and $l_{ii}(t)=-\sum_{j=1}^{m}l_{ij}(t)$. The time-varying in-neighborhood of node $i$ is $N_{i}(t)=\{j:~a_{ij}(t)\ne 0\}$, and similarly
\begin{align}
\Lambda_{ij}(t)=\{k: a_{ik}(t)>0 \; {\rm and}\;a_{jk}(t)>0\}.  \label{Lam}
\end{align}

We are interested in the dynamics of the phase differences (PDs)
$\theta_{ij}(t):=\theta_{i}(t)-\theta_{j}(t)$ between the oscillators.
Clearly only $m(m-1)/2$ of these quantities are independent
since $\theta_{ii}=0$ and $\theta_{ij}=-\theta_{ji}$.
As a shorthand notation, the collection of phase differences will be denoted by the corresponding uppercase symbol, i.e., $
\Theta=\{\theta_{ij}:i>j; \; i,j=\onetom\} \in \R^{m(m-1)/2}$,
and $\|\Theta\|$ then refers to the norm of in $\R^{m(m-1)/2}$.
Correspondingly, we talk about \emph{phase difference regions} $\mathscr{A}$ in $\R^{m(m-1)/2}$,
but also regard them as a subset of $\R^{m^2}$ subject to the constraints mentioned above. The following definition is a generalization of the phase-locking dynamics of the standard Kuramoto model \eqref{Kura}, as extended to \eqref{TVKR}.
Similar to $\theta_{ij}$, the notation $\phi_{ij}(t)=\phi_{i}(t)-\phi_{j}(t)$ stands for the phase differences for any solution $\phi(t)=\{ \phi_{i}(t)\}_{i=1}^{m}$ of \eqref{TVKR}. Motivated by the asymptotic stability of trajectories, also known as attractive trajectory and extreme stability \cite{LS}, we present the following definition.

\begin{definition}
The PD trajectories of the system (\ref{TVKR}) is said to be \emph{asymptotically stable within a phase difference region $\mathscr A\subset\R^{m\times (m-1)/2}$} if for any two solutions $\phi(t)$ and $\theta(t)$ of (\ref{TVKR}), the phase differences satisfy $\lim\limits_{t\to\infty}|\phi_{ij}(t)-\theta_{ij}(t)|=0$ for all $i,j$ whenever the initial conditions $\{\phi_{ij}(0)\}_{i>j},\{\theta_{ij}(0)\}_{i>j}$ belong to $\mathscr A$.
In addition, if the convergence is exponential, i.e., there exist positive constants $M$, $T$, and $\epsilon$ such that
\begin{align*}
|\phi_{ij}(t)-\theta_{ij}(t)|\le M\max_{i,j}|\phi_{ij}(0)-\theta_{ij}(0)|\exp(-\epsilon t)
\end{align*}
for  $t\ge T$ and all $i,j=\onetom$, then the PD trajectories of system (\ref{TVKR}) is said to be \emph{exponentially asymptotically stable within a phase difference region $\mathscr A$}.
\end{definition}

In this paper, for $r\in[0,\pi/2)$, we consider PD regions of the form
\begin{equation*}
\mathscr{A}^{r}=\{\theta_{ij} : |\theta_{ij}|\le r,~i>j\}.
\end{equation*}
Guaranteeing that the phase trajectory $\theta(t)$ starting from within $\mathscr{A}^{r}$ stays inside $\mathscr{A}^{r}$ requires some conditions, such as those given in the next lemma.

\begin{lemma}\label{lem0}
The set $\mathscr{A}^{r}$ is invariant for system \eqref{TVKR}
if
\begin{align}
&\omega_{i}(t)-\omega_{j}(t)-[a_{ij}(t)
+a_{ji}(t)]\sin(r)-\sum_{k\notin \Lambda_{ij}(t),k\ne i,j}\{[a_{ik}(t)]^{-}+[a_{jk}(t)]^{-}\}\sin(r)\nonumber\\
&-\sum_{k\in \Lambda_{ij}(t)}\min\{a_{ik}(t),a_{jk}(t)\}\sin(r)<0  \label{invariant}
\end{align}
for all $i\ne j$ and $t\ge 0$,
where $[a_{ij}(t)]$ is the time-varying weighted adjacency matrix and  $\Lambda_{ij}(t)$ is defined in \eqref{Lam}.
\end{lemma}

This lemma is proved in Appendix A.
The following result follows by the lemma and is useful towards a robust condition for the invariance of $\mathscr A^{r}$, when, for instance, the time variation is not exactly known due to noise, unknown failures, or uncertainties.

\begin{proposition}
The set $\mathscr{A}^{r}$ is invariant for system \eqref{TVKR} if
\begin{align}\label{invar2}
\frac{\Delta\omega}{\sin(r)}\le\mu_{0}+\mu_{2}-\mu_{1},
\end{align}
where $\Delta\omega=\sup_{t}\max_{i,j}|\omega_{i}(t)-\omega_{j}(t)|$ is the maximum frequency displacement,
$$\mu_{0}=\sup_{t}\min_{i,j}\sum_{k\in\Lambda_{ij}(t)}\min\{a_{ik}(t),a_{jk}(t)\}$$
is the minimum value of ergodic coefficient of the graphs with the adjacency matrix $[a^{+}_{ij}(t)]$,
$$\mu_{1}=\sup_{t}\max_{i,j}\sum_{k\notin \Lambda_{ij}(t),k\ne i,j}\{-[a_{ik}(t)]^{-}-[a_{jk}(t)]^{-}\},$$ and $\mu_{2}=\sup_{t}\min_{i,j}[a_{ij}(t)+a_{ji}(t)]$.
\end{proposition}

Both conditions (\ref{invariant}) and (\ref{invar2}) are in fact rather conservative. In the following, we assume the invariance of $\mathscr{A}^{r}$ and validate it through simulations.

We first consider the case of the nonnegative coupling coefficients and have the following result immediately as a consequence from \cite{LLC2013}.

\begin{theorem}\label{thm1}
Assume $a_{ij}(t) \ge 0$, $\forall i\ne j$ and $t\ge 0$. Suppose that $\mathscr{A}^{r}$ is invariant for \eqref{TVKR} for some $r\in[0,\pi/2)$.
Then the PD trajectories of system (\ref{TVKR}) are asymptotically stable within $\mathscr{A}^{r}$
if there exist $T>0$, and sequences $0=t_{1}< t_{2}<\cdots< t_{n}<\cdots$ and $\eta_{n}>0$, $n\in \mathbb{N},$ with $\sum_{n=1}^{\infty}\eta_{n}=+\infty$, such that when each time interval $[t_{n},t_{n+1}]$ is partitioned into $m-1$ time bins with
$$t_{n}=t_{n}^{0}<t_{n}^{1}<\cdots<t_{n}^{m-1}=\min\{t_{n},t_{n-1}+T\},$$
then the $\eta_{n}$-graph corresponding to the Laplacian matrix $Z^{k,n}=[z^{k,n}_{ij}]$ with components
\begin{align*}
z^{k,n}_{ij}= -\int_{t_{n}^{k}}^{t_{n}^{k+1}}a_{ij}(s)\,ds, \; i\ne j; \qquad
z^{k,n}_{ii} = -\sum_{j\ne i}z^{n}_{ij}(t),
\end{align*}
has a spanning tree for all $k=1,\dots,m-1$ and $n\in \mathbb{N}$.
\end{theorem}



\begin{proof}
Consider two solutions $\theta, \phi$ of (\ref{TVKR}).
The differences $\delta_{i}=\phi_{i}-\theta_{i}$ between the two solutions evolve by the equations
\begin{equation*}
\dot{\delta}_{i}=\sum_{j=1}^{m}a_{ij}(t)[\sin(\phi_{j}(t)-\phi_{i}(t))
-\sin(\theta_{j}(t)-\theta_{i}(t))].
\end{equation*}
Denoting the phase differences by  $\phi_{ij}=\phi_{i}-\phi_{j}$ and $\theta_{ij}=\theta_{i}-\theta_{j}$,
by the mean value theorem there exist numbers
$\zeta_{ij}\in[\min(\theta_{ij},\phi_{ij}),\max(\theta_{ij},\phi_{ij})]$ such that
\begin{equation}
\dot{\delta}_{i}=\sum_{j=1}^{m}[a_{ij}(t)\cos(\zeta_{ji}(t))](\delta_{j}-\delta_{i}),
\quad i=\onetom.\label{linear}
\end{equation}
It can be seen that $\cos\zeta_{ij}=\cos\zeta_{ji}$ for all $i,j$.
Let $B(t)=[b_{ij}(t)]$ with $b_{ij}(t)=-a_{ij}(t)\cos(\zeta_{ji}(t))$, which is nonpositive for $i\ne j$ and $b_{ii}(t)=-\sum_{j\ne i}b_{ij}(t)$. Since $\{\theta_{ij}(t)\}_{i>j}$ and $\{\phi_{ij}(t)\}_{i>j}$ belong to $\mathscr{A}^{r}$ for all $t\ge 0$, we have $\{\zeta_{ij}(t)\}_{i>j}\in\mathscr{A}^{r}$  for all $t\ge 0$. Hence, the $\eta_{n}\cos(r)$-graph with the Laplacian $\int_{t_{n}^{k}}^{t_{n}^{k+1}}B(s)ds$ has a spanning tree for all $k=1,\dots,m-1$ and $n\in \mathbb{N}$.
Theorem 1 in \cite{LLC2013} shows that
\begin{equation*}
\lim\limits_{t\to\infty}|\delta_{i}(t)-\delta_{j}(t)|=0, \quad \forall~i,j=\onetom.
\end{equation*}
Note that $$\phi_{ij}-\theta_{ij}=\phi_{i}-\phi_{j}-\theta_{i}+\theta_{j}=\delta_{i}-\delta_{j}.$$ Hence, $\lim_{t\to\infty}[\phi_{ij}(t)-\theta_{ij}(t)]=0$ for all $i,j$, which completes the proof.
\end{proof}

The following corollary is a direct consequence of Theorem \ref{thm1}; see also \cite[Theorem 1]{Mor} and \cite[Theorem 31]{YLC2013}.

\begin{corollary}\label{cor1}Assume $a_{ij}(t) \ge 0$, $\forall i,j$. Let $r\in[0,\pi/2)$, and suppose that $\mathscr{A}^{r}$ is invariant for \eqref{TVKR}.
Then the PD trajectories of \eqref{TVKR} are exponentially asymptotically stable within $\mathscr{A}^{r}$
if there exist $T>0$ and $\eta>0$ such that the $\eta$-graph corresponding to the Laplacian matrix $Z(t)=[z_{ij}(t)]$ with components
\begin{align}\label{Z}
z_{ij}(t)=\begin{cases}
-\int_{t}^{t+T}a_{ij}(s)\,ds, & i\ne j,\\
-\sum_{j\ne i}z_{ij}(t), & i=j,
 \end{cases}
 \end{align}
has a spanning tree for all $t\ge 0$.
\end{corollary}

\begin{remark}
Theorem \ref{thm1} and Corollary \ref{cor1} can be further extended, for instance to the case when there are two or more disjoint node subsets such that in the subgraph of each subset the conditions of Theorem 1 hold. Then one can conlude that in each node subset the PD trajectories are asymptotically stable; however, the PDs between oscillators in different node subset may fail to be stable.
\end{remark}

We next consider the case when the coupling coefficients $a_{ij}(t)$ are allowed to have negative values. We can prove the following result by employing a matrix measure similar to the one proposed in \cite{LC2008}.

\begin{theorem}\label{thm2}
Let $r\in[0,\pi/2)$, and suppose that $\mathscr{A}^{r}$ is invariant for \eqref{TVKR}. Let
\begin{eqnarray}
c_{ij}^{r}(t)&=&\begin{cases}[a_{ij}(t)+a_{ji}(t)]\cos(r)&a_{ij}(t)+a_{ji}(t)>0\\
a_{ij}(t)+a_{ji}(t)& a_{ij}(t)+a_{ji}(t)\le 0\end{cases},\nonumber\\
\tilde{a}_{ij}^{r}&=&\begin{cases}a_{ik}(t)\cos(r)& a_{ik}(t)>0\\
a_{ik}(t)& a_{ik}(t)\le0\end{cases}.\label{Lr}
\end{eqnarray}
Define the index
\begin{eqnarray}
\xi(L(t),r) := -\min_{i\ne j}\{c_{ij}^{r}(t)+\sum_{k\ne i,j}\min(\tilde{a}_{ik}^{r}(t),\tilde{a}_{jk}^{r}(t))\}.  \label{xi}
\end{eqnarray}
If there exists $T>0$ and $\eta>0$ such that $(1/T)\int_{t+T}^{T}\xi(L(s),r)ds\le- \eta$ for all $t$,
 then  the PD trajectories of the coupled system \eqref{TVKR} are exponentially asymptotically stable within $\mathscr{A}^{r}$.
\end{theorem}
\begin{proof}
Define $V(\delta)=\max_{i}\delta_{i}-\min_{j}\delta_{j}$. Let $\delta_{i}(t)$ be a solution of (\ref{linear}).
For any $t\ge 0$, let $i^{*}$ be any index such that $\delta_{i^{*}}=\max_{i}\delta_{i}(t)$, and similarly, $i_{*}$ be any index such that $\delta_{i_{*}}=\min_{i}\delta_{i}(t)$.
Note that $i_{*}$ and $i^{*}$ are time-varying. Then,
\begin{eqnarray*}
&&\frac{d[\delta_{i^{*}}-\delta_{i_{*}}]}{d\tau}\left|_{\tau=t}\right.\\
&&=\sum_{j=1}^{m}a_{i^{*}j}\cos(\zeta_{i^{*}j}(t))[\delta_{j}(t)-\delta_{i^{*}}(t)]-
\sum_{k=1}^{m}a_{i_{*}k}\cos(\zeta_{i_{*}k}(t))[\delta_{k}(t)-\delta_{i_{*}}(t)]\\
&&=(-a_{i^{*}i_{*}}-a_{i_{*}i^{*}})\cos(\zeta_{i^{*}i_{*}}(t))
[\delta_{i^{*}}(t)-\delta_{i_{*}}(t)]\\
&&-\sum_{j\ne i_{*},i^{*}}a_{i^{*}j}\cos(\zeta_{i^{*}j}(t))[\delta_{i^{*}}(t)-\delta_{j}(t)]-\sum_{k\ne i_{*},i^{*}}a_{i_{*}k}\cos(\zeta_{i_{*}j}(t))[\delta_{k}(t)-\delta_{i_{*}}(t)]\\
&&\le-c_{ij}^{r}(t)[\delta_{i^{*}}(t)-\delta_{i_{*}}(t)]-\sum_{j\ne i^{*},i_{*}}\tilde{a}_{i^{*}j}^{r}[\delta_{i^{*}}(t)-\delta_{j}(t)]-
\sum_{k\ne i^{*},i_{*}}\tilde{a}_{i_{*}k}^{r}[\delta_{k}(t)-\delta_{i_{*}}(t)]\\
&&\le-\left\{c_{ij}^{r}(t)+\sum_{j\ne i^{*},i_{*}}\min(\tilde{a}_{i^{*}j}^{r},\tilde{a}_{i_{*}j}^{r})\right\}
[\delta_{i^{*}}(t)-\delta_{i_{*}}(t)]\le\xi(L(t),r)V(\delta(t)),
\end{eqnarray*}
where the $\zeta_{ij}$ are defined in (\ref{linear}) and satisfy $\cos(\zeta_{ij})=\cos(\zeta_{ji})$. Since the above holds for all $i^{*}$ and $i_{*}$ that pick the maximum and minimum of $\delta_{i}(t)$, we have
\begin{eqnarray*}
\frac{dV(\delta(\tau))}{d\tau}\left|_{\tau=t}\right.\le \xi(L(t),r)V(\delta(t)),
\end{eqnarray*}
which implies
\begin{eqnarray*}
V(\delta(t))\le\exp\left(\int_{0}^{t}\xi(L(s),r)ds\right)V(\delta(0)).
\end{eqnarray*}
%
The condition $(1/T)\int_{t+T}^{T}\xi(L(s),r)ds\le- \eta$ implies that $\int_{0}^{\infty}\xi(L(s),r)ds=-\infty$. Therefore, $\lim_{t\to\infty}V(\delta(t))=0$ holds uniformly, and so $\lim_{t\to\infty}[\delta_{i}(t)-\delta_{j}(t)]=0$ uniformly for all $i,j$.
The proof is completed by the same arguments as in the proof of Theorem \ref{thm1}.
\end{proof}

\begin{remark}
By the transformation (\ref{linear}), the asymptotic stability of the phase difference trajectories corresponds to the synchronization of the time-varying system (\ref{linear}). The method of the proof of Theorem \ref{thm2} is analogous to the Hajnal diameter approach used in \cite{LAJ1,LAJ2,YLC2013}.
\end{remark}

Finally, we consider the case that $L(t)$ is symmetric and positive semidefinite but some elements $a_{ij}(t)$ ($i\ne j$) may be negative. To this end, we need the following lemmas.

\begin{lemma}\label{lem1}
Let $L=[l_{ij}]_{i,j=1}^{m}\in\R^{m,m}$ be a symmetric square matrix satisfying (i) all row sums equal to zero, and (ii) zero is a simple eigenvalue. Let $r\in[0,\pi/2)$ and define the matrix $\tilde{L}^{r}=[\widetilde{l}^{r}_{ij}]_{i,j=1}^{m}$ by
\begin{eqnarray}
\widetilde{l}^{r}_{ij}=\begin{cases}l_{ij}\cos(r), & i\ne j,~l_{ij}\le 0\\
l_{ij}, & i\ne j,~l_{ij}>0\\
-\sum_{k\ne i}\widetilde{l}_{ik}^{r}, & i=j\end{cases}.\label{tilder}
\end{eqnarray}
Let $\chi_{1}$ and $\chi_{2}$ denote the smallest eigenvalues of $L$ and $\widetilde{L}^{r}$, respectively, over the eigenspace orthogonal to ${\bf 1}=[1,\dots,1]^{\top}\in \mathbb{R}^{m}$. Then
$
\chi_{1}\ge\chi_{2}.
$
\end{lemma}

The proof of this lemma is given in Appendix B.

\begin{lemma} \label{thmX}
Let $G(t)$ be a symmetric matrix with piecewise continuous elements,
such that $G(t)$ is positive semidefinite, has all row sums equal to zero, and there exists $R>0$ such that $\|G(t)\|\le R$, $\forall t\ge 0$.
Let $\bar{G}(t,s)=(1/(t-s))\int_{s}^{t}G(\tau)d\tau$ and define $\widetilde{\bar{G}}^{r}(t,s)$ analogously to \eqref{tilder}.
Suppose there exists $h>0$ such that $\sum_{k=0}^{\infty}\beta_{k}=+\infty$, where
\begin{eqnarray*}
\beta_{k}=\lambda_{2}(\tilde{\bar{G}}^{r}((k+1)h, kh)).
\end{eqnarray*}
Then the time-varying linear system
\begin{eqnarray}
	\dot{x}=-G(t)x  \label{ma}
\end{eqnarray}
reaches consensus, i.e., $\lim_{t\to\infty}[x_{i}(t)-x_{j}(t)]=0$ \ $\forall i,j$, where the $x_{i}$ denote the components of $x \in \mathbb{R}^{m}$.
If, in addition,  there exists $\hat{\beta}>0$ such that $\beta_{k}>\hat{\beta}$ for all $k\in\mathbb{Z}^+$, then the convergence is exponential.
\end{lemma}

The proof of this lemma is given in Appendix C.

Thus, when $L(t)$ is symmetric and positive semidefinite, the next result follows. 

\begin{theorem}\label{thm3}
Suppose that $\mathscr{A}^{r}$ is invariant for the system \eqref{TVKR} for some $r\in[0,\pi/2)$ and $L(t)$ is symmetric and positive semidefinite for all $t\ge 0$. Let $\bar{L}(t,s)=(1/(t-s))\int_{s}^{t}L(\tau)d\tau$ and
\begin{eqnarray*}
\alpha_{k}(h)=\lambda_{2}(\tilde{\bar{L}}^{r}(kh,(k+1)h))
\end{eqnarray*}
for some $h>0$, where $\tilde{\bar{L}}^{r}(kh,(k+1)h)$ is defined analogously to \eqref{tilder}.
If there exists $h>0$ such that
$\sum_{k=0}^{\infty}\alpha_{k}(h)=+\infty$, then the PD trajectories of the system \eqref{TVKR} are asymptotically stable within $\mathscr A^{r}$.
\end{theorem}

\begin{proof} Consider the equations (\ref{linear}). Let $B(t)=[b_{ij}(t)]$, with $b_{ij}(t)=-a_{ij}(t)\cos(\zeta_{ji}(t))$ if $i\ne j$ and $b_{ii}(t)=-\sum_{j\ne i}b_{ij}(t)$. Let $\varrho_{k}$ be the smallest eigenvalue of $\left(\frac{1}{h}\int_{kh}^{(k+1)h}B(s)ds\right)$ over the eigenspace orthogonal to $\mathbf 1$.
Since $L(t)$ is symmetric and positive semidefinite, so is $B(t)$. By Lemma \ref{lem1},
\begin{eqnarray*}
\varrho_{k}=\lambda_{2}\left(\frac{1}{h}\int_{kh}^{(k+1)h}B(s)ds\right)\ge \lambda_{2}
\left(\tilde{\bar{B}}^{r}((k+1)h,kh)\right)=\alpha_{k}\ge 0.
\end{eqnarray*}
Then, $\sum_{k=0}^{\infty}\varrho_{k}\ge\sum_{k=0}^{\infty}\alpha_{k}=+\infty,$ and Theorem \ref{thm3} follows by Lemma \ref{thmX}.
\end{proof}

Moreover, we have the following result on exponential asymptotic stability.
\begin{corollary}
\label{cor2}
Under the hypotheses and notations in Theorem \ref{thm3},
if there exist $h>0$ and $\hat{\alpha}>0$ such that
\begin{eqnarray}
\alpha_{k}=\lambda_{2}(\tilde{\bar{L}}^{r}(kh,(k+1)h))>\hat{\alpha},\label{alpha2}
\end{eqnarray}
then the PD trajectories of the coupled system \eqref{TVKR} are exponentially asymptotically stable within $\mathscr A^{r}$.
\end{corollary}


\section{Asymptotic dynamics of phase differences}
In this section, we investigate time-varying patterns of phase differences in (\ref{TVKR}) in three common scenarios.

\subsection{Asymptotic periodicity}

\begin{definition}
A vector-valued function $x(t)\in\R^{n}$ is said to be \emph{asymptotically periodic (AP) with period $T$} if there exists a $T$-periodic function $x^{*}(t)$ such that $\lim_{t\to\infty}\|x(t)-x^{*}(t)\|=0$. In addition, if the convergence is exponential, then $x(t)$ is said to be \emph{exponentially asymptotically periodic} (EAP).
\end{definition}

Consider the following hypothesis:

${\bf H}_1$: $\omega_{i}(t)$ and $a_{ij}(t)$ are piecewise continuous and periodic with a fixed period $T$ for all $i,j=\onetom$.

Then we have the following result.

\begin{proposition}\label{prop1}
Assume the hypotheses  ${\bf H}_1$, let $r\in[0,\pi/2)$, and suppose that $\mathscr{A}^{r}$ is invariant for \eqref{TVKR}.
Then the PD trajectories of \eqref{TVKR} starting from initial values in $\mathscr A^{r}$ are exponentially asymptotically periodic provided any one of the following conditions holds:
\begin{enumerate}
\item $a_{ij}(t)\ge 0$ for all $i\ne j$ and $t\ge 0$, and there exist $T>0$ and $\eta>0$ such that the $\eta$-graph corresponding to the Laplacian matrix $[\int_{0}^{T}-a_{ij}(s)ds]_{i,j=1}^{m}$ with $a_{ii}(t)=-\sum_{j\ne i}a_{ij}(t)$ has a spanning tree;
\item $a_{ij}(t)\in\mathbb R$ for all $i,j$ and $t\ge 0$,  there exists $\eta>0$ such that $(1/T)\int_{0}^{T}\xi(L(s),r)ds\le- \eta$, where $\xi(L(t),r)$ is as defined in \eqref{xi};
\item $a_{ij}(t)\in\mathbb R$ for all $i,j$ and $t\ge 0$, $L(t)$ is symmetric and positive semidefinite for all $t\in[0,T]$, and the inequality \eqref{alpha2} holds.
\end{enumerate}

\end{proposition}


\begin{proof}
On the basis of Hypothesis ${\bf H}_{1}$, condition 1 implies that the $\eta$-graph corresponding to the Laplacian matrix $Z(t)$ defined in \eqref{Z}
has a spanning tree for all $t\ge 0$.
Condition 2 implies that $(1/T)\int_{t}^{t+T}\xi(L(s),r)ds\le- \eta$ for all $t\ge 0$. Finally, condition 3 implies that $L(t)$ is symmetric and positive semidefinite for all $t\ge 0$.
Thus, under any one of these conditions, Corollary \ref{cor1}, Theorem \ref{thm2}, and Corollary \ref{cor2} guarantee that the PD trajectories are  exponentially stable.

Let $\Theta(t)=[\theta_{ij}(t)]_{i>j}$ and $\Phi(t)=[\phi_{ij}]_{i>j}$ be the phase differences of the solutions of (\ref{TVKR})
with initial values such that $\Theta(0)=[\theta_{ij}(0)]_{i>j}$ and $\Phi(0)=[\phi_{ij}(0)]_{i>j}$. We have
\begin{equation}
\|\Theta(t)-\Phi(t)\|\le M\|\Theta(0)-\Phi(0)\|\exp(-\epsilon t)\label{as1}
\end{equation}
for some $M$ and $\epsilon>0$.
Consider the mapping
\begin{equation*}
H : \Omega \to\Omega, \quad
\Theta(0) \mapsto\Theta(T),
\end{equation*}
where $\Omega $ is the compact hypercube in $\R^{m(m-1)/2}$
given by
\begin{equation*}
\Omega=\left\{[\theta_{ij}]_{i>j}: |\theta_{ij}|\le r,
 \; i,j=\onetom\right\}.
\end{equation*}
%
We will show that the mapping $H$ is well defined. Let two initial values $\theta(0)$ and $\vartheta(0)$ belonging to $\mathscr{A}^{r}$ be given such that $\theta_{i}(0)-\theta_{j}(0)=\vartheta_{i}(0)-\vartheta_{j}(0)=\theta_{ij}(0)$ for all $i>j$, which implies $\theta_{ij}(0)\in[-r,r]$, and there exists a unique $\theta_{0}$ such that $\vartheta_{i}(0)=\theta_{i}(0)+\theta_{0}$ for all $i=\onetom$. Let the solution of (\ref{TVKR}) with these initial values be denoted by $\theta_{i}(t)$ and $\vartheta_{i}(t)$, respectively, which are still contained in $\mathscr{A}^{r}$.

Let $\psi_{i}(t)=\theta_{i}(t)+\theta_{0}$ for all $i=\onetom$. Noting that
\begin{align*}
\dot{\psi}_{i}=\dot{\theta}_{i}&=\omega_{i}(t)+\sum_{j=1}^{m}a_{ij}(t)\sin\left[(\theta_{j}
+\theta_{0})
-(\theta_{i}+\theta_{0})\right]\\
&=\omega_{i}(t)+\sum_{j=1}^{m}a_{ij}(t)\sin\left[\psi_{j}-\psi_{i}\right],\quad i=\onetom,
\end{align*}
one can see that $\{\psi_{i}(t):i=\onetom\}$ are solutions of (\ref{TVKR}) with initial values $\psi_{i}(0)=\vartheta_{i}(0)$. By the uniqueness of the solution, $\vartheta_{i}(t)=\psi_{i}(t)=\theta_{i}(t)+\theta_{0}$ for all $t\ge 0$ and $i$. Hence, $\vartheta_{ij}(t)=\theta_{ij}(t)$ for all $i,j$. Therefore, given the initial values of PD, $\Theta(0)\in\Omega$, the PD $\Theta(t)\in\Omega$ of the solutions of (\ref{TVKR}) exists and is unique; that is, the mapping $H$ is well-defined.

Thus, there exists an integer $K$ such that $M\exp(-\epsilon TK)<1$. Then
\begin{equation*}
\|H^{(K)}\circ\Theta(0)-H^{(k)}\circ\Phi(0)\|\le M\exp(-\epsilon TK) \, \|\Theta(0)-\Phi(0)\|,
\end{equation*}
which implies that $H^{(K)}$ is a contraction map. Hence, there exists a unique fixed point $\Theta^{*}=[\theta^{*}_{ij}]_{i>j}$ of  $H^{(K)}$, namely, $H^{(K)}(\Theta^{*})=\Theta^{*}$. Note that $H(\Theta^{*})$ is still a fixed point of $H^{(K)}$. By the uniqueness of the fixed point of the contraction map, $H(\Theta^{*})=\Theta^{*}$. Hence, $\theta_{ij}^{*}(0)=\theta_{ij}^{*}(T)$. Consider the solution $\theta^{*}_{i}(t)$ of (\ref{TVKR}) with $\theta_{i}^{*}(0)-\theta_{j}^{*}(0)=\theta^{*}_{ij}(0)$. Under the hypotheses ${\bf H}_1$, we have $\theta_{ij}^{*}(t+T)=\theta_{ij}^{*}(t)$ for all $t\ge 0$ and $i,j=\onetom$. That is, $\theta_{ij}^{*}(t)=\theta_{i}^{*}(t)-\theta^{*}_{j}(t)$ is periodic with period $T$.
Combined with the conditions of Corollary \ref{cor1}, and Theorems \ref{thm2} and \ref{thm3}, this periodic PD trajectory is exponentially asymptotically stable, which completes the proof.
\end{proof}

As a numerical example, we consider a network of five Kuramoto oscillators with periodical switching between two coupling matrices $A$ and two intrinsic frequencies $\omega$ as follows:
\begin{eqnarray*}
\omega(t)=\begin{cases}\omega^{1}, & t\in[(2k-1)T,2kT)\\
\omega^{2}, & t\in[2kT,(2k+1)T)\end{cases} \quad L(t)=
\begin{cases}L^{1}, & t\in[(2k-1)T,2kT)\\
L^{2}, & t\in[2kT,(2k+1)T)\end{cases}
\end{eqnarray*}
with $k\in \mathbb{N}$, where $T=2$ (sec) and
\begin{eqnarray*}
&&\omega^{1}=[ 0.1294, 1.9765,  1.8790,   0.7331,    1.1332]^{\top},\\
&&\omega^{2}=
[1.9578,    0.5295,    1.1234,    1.3591,    2.1786]^{\top},\\
&&L^{1}=\left[\begin{array}{rrrrr}
-4.5343  &  0.5795   & 1.7331  &  0.9795  &  1.2422\\
    0.2241 &  -1.9971  &  0.4334 &   0.2703  &  1.0692\\
    1.6323 &  -0.0286  & -3.5243 &   1.2298  &  0.6908\\
    0.1402 &   0.7296  &  0.6795 &  -2.4363  &  0.8870\\
    0.5957 &   0.4723  &  0.4909 &  -0.0119 &  -1.5470
\end{array}\right],\\
&&L^{2}=\left[\begin{array}{rrrrr}
-4.6960  &  0.3018  &  1.7915  &  1.6922  &  0.9105\\
    0.1732 &  -2.3331  &  1.6492  &  0.3674  &  0.1433\\
    1.0687 &   0.4723  & -3.1947  &  0.5293   & 1.1245\\
    0.7140  &  1.1625  &  0.0833  & -3.1210  &  1.1612\\
    1.3104  &  0.1241  &  1.3107  &  1.1887 &  -3.9339\\
\end{array}\right].
\end{eqnarray*}
(The components of $L^{1,2}$ and $\omega^{1,2}$ are randomly generated until the specific criteria of Proposition \ref{prop1} are met.)
Taking $r=\pi/3$, we calculate  $\xi(L^{1},r)= 0.0858$ and $\xi(L^{2},r)=-0.1249$. Thus,
$\xi(L^{1},r)+\xi(L^{2},r)=-0.0391 < 0$; hence, the conditions of Theorem \ref{thm2} and Proposition \ref{prop1} hold.

As shown in Fig.~\ref{Fig1_AP}, the phase differences are asymptotically stable and converge to periodic trajectories. In addition, $\mathscr{A}^{r}$ (with $r=\pi/3$) is indeed found to be invariant for (\ref{TVKR}).


\begin{figure}[htp]
\centering
\includegraphics[width=.9\textwidth]
{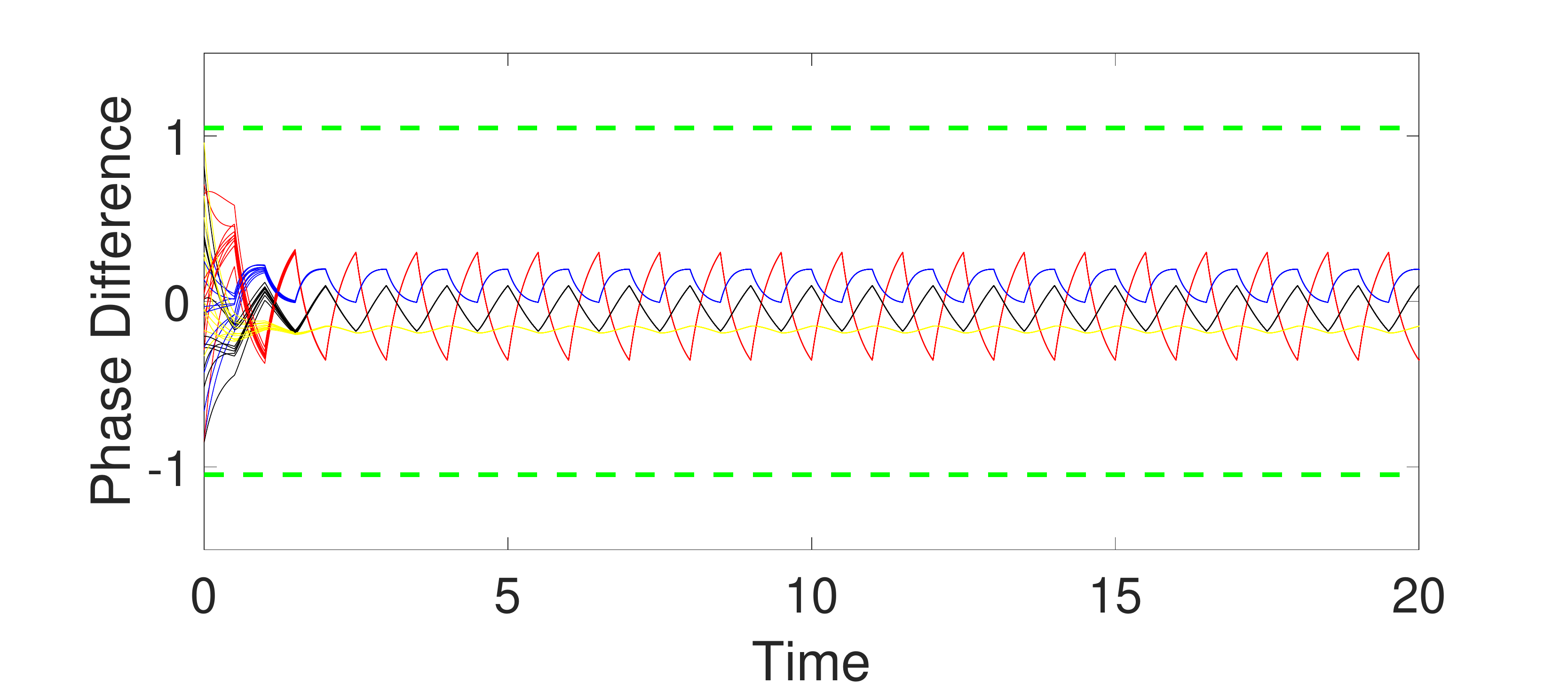}
\caption{Dynamics of the phase differences: $\theta_{1}(t)-\theta_{2}(t)$ (red), $\theta_{2}(t)-\theta_{3}(t)$ (black), $\theta_{3}(t)-\theta_{4}(t)$ (blue), and $\theta_{4}(t)-\theta_{5}(t)$ (yellow) of ten simulations with randomly chosen initial values from $[-\pi/6,\pi/6]$ following a uniform distribution. The two horizontal green dashed lines mark the values $\pm\pi/3$ corresponding to $\pm r$.}
\label{Fig1_AP}
\end{figure}



\subsection{Small perturbations}

For a small parameter $\epsilon$, consider the following hypothesis:

 ${\bf H}_{2}$: The frequencies and coupling strengths have the form
\begin{equation}
\omega_{i}(t)=\bar{\omega}_{i}+\epsilon\Omega_{i}(t),\qquad a_{ij}(t)=\bar{a}_{ij}+\epsilon A_{ij}(t),
\label{sp}
\end{equation}
where $\Omega_{i}(t)$ and $A_{ij}(t)$ are piecewise continuous, bounded, and periodic functions with period $T$ such that
\begin{equation}
\int_{0}^{T}\Omega_{i}(t)dt=0,\quad \int_{0}^{T}A_{ij}(t)dt=0.\label{periodic}
\end{equation}

Let $\bar{\theta}_{ij}$, with $|\bar{\theta}_{ij}|\in[0,\pi/2]$ for all $i,j$, be constant PDs of the the phase-locked solution of the following system with static parameters:
\begin{eqnarray}
\dot{\bar{\theta}}_{i}=\bar{\omega}_{i}+\sum_{j=1}^{m}\bar{a}_{ij}\sin(\bar{\theta}_{j}
-\bar{\theta}_{i}),\quad i=\onetom.
\label{SKR}
\end{eqnarray}
Namely, there exist $\Omega>0$ and $\bar{\vartheta}_{ij}\in[0,2\pi)$ with $\bar{\theta}_{ij}=\bar{\vartheta}_{i}-\bar{\vartheta}_{j}$,
such that
\begin{equation}
\bar{\theta}_{i}(t)=\Omega t+\bar{\vartheta}_{i}, \quad i=\onetom.\label{ps}
\end{equation}
For $\epsilon\to 0$, we consider a perturbation solution of (\ref{TVKR}) in the form
\begin{equation}
\theta_{i}(t)=\bar{\theta}_{i}(t)+\epsilon\Phi_{i}(t)+o(\epsilon). \label{represent}
\end{equation}
Differentiating both sides of (\ref{represent}) and comparing terms of first order in $\epsilon$ gives
\begin{equation}
\dot{\Phi}_{i}=\Omega_{i}(t)+\sum_{j=1}^{m}A_{ij}(t)\sin(\bar{\theta}_{ji})+\sum_{j=1}^{m}
\bar{a}_{ij}\cos(\bar{\theta}_{ji})[\Phi_{j}-\Phi_{i}].\label{variation}
\end{equation}

\begin{proposition}  \label{propx}
Let $r\in[0,\pi/2)$
and suppose that $\mathscr{A}^{r}$ is invariant in (\ref{TVKR}), the hypotheses ${\bf H}_{2}$ hold, and \eqref{SKR} possesses a phase-locked solution $[\bar{\theta}_{1}(t),\dots,\bar{\theta}_{m}(t)]^{\top}\in\mathscr{A}^{r}$ as described by \eqref{ps}.
Suppose further that any one of the following conditions holds:
\begin{enumerate}
\item $a_{ij}(t)\ge 0$ for all $i\ne j$ and $t\ge 0$, and
the graph corresponding to the Laplacian $\bar{L}=[\bar{L}_{ij}]$ with
\begin{equation*}
\bar{L}_{ij}=-\bar{a}_{ij}, \; i\ne j; \qquad
\bar{L}_{ii} = -\sum_{j=1}^{m}\bar{l}_{ij},
\end{equation*}
has a spanning tree;
\item $\xi(\bar{L},r)<0$;
\item $L(t)$ is symmetric and positive semidefinite for all $t\ge 0$, and $\lambda_{2}(\bar{L})>0$.
\end{enumerate}
Then there exist $U>0$ and $\Phi_{i}(t)$ satisfying $|\Phi_{i}(t)|<U$ for all $i$ and $t$, such that \eqref{TVKR} has a solution in the form of $\theta_{i}(t)=\bar{\theta}_{i}(t)+\epsilon\Phi_{i}(t)+o(\epsilon)$ as $\epsilon\to 0$. Furthermore, if $\epsilon$ is sufficiently small, the PD trajectories $\theta_{ij}(t)=\theta_{i}(t)-\theta_{j}(t)$ are asymptotically stable with $\mathcal A^{r}$.
\end{proposition}

\begin{proof}
We first show that the $\Phi_{i}(t)$ are bounded. Let $Y=[y_{ij}]$, where $y_{ij}=\bar{a}_{ij}\cos(\bar{\theta}_{ji})$ for $i\ne j$ and $y_{ii}=-\sum_{j=1}^{m}y_{ij}$, and
\begin{eqnarray*}
z_{i}(t)=\Omega_{i}(t)+\sum_{j=1}^{m}A_{ij}(t)\sin(\bar{\theta}_{ji}),\quad i=\onetom.
\end{eqnarray*}
Then we can rewrite (\ref{variation}) in the compact form
\begin{eqnarray}
\dot{\Phi}=z(t)+Y\Phi(t),\label{variation1}
\end{eqnarray}
where $z(t)=[z_{1}(t),\dots,z_{m}(t)]^{\top}$ and $\Phi(t)=[\Phi_{1}(t),\dots,\Phi_{m}(t)]^{\top}$.
The solution of (\ref{variation1}) is
\begin{eqnarray}
\Phi(t)=\exp(Yt)\Phi(0)+\int_{0}^{t}\exp(Y(t-s))z(s)ds.\label{sol1}
\end{eqnarray}
We shall prove that $\|\Phi(t)\|$ is bounded by some constant for all $t\ge 0$. To this end, we require the following lemma.

\begin{lemma}\label{lem3}
Any one of conditions 1, 2 and 3 of Theorem~\ref{propx} implies that $Y$ has a simple zero eigenvalue and all other eigenvalues have negative real parts.
\end{lemma}


See Appendix D for a proof. This lemma implies that the first term $\|\exp(Yt)\Phi(0)\|$ is bounded for $t\ge 0$.
%
%
We write $Y=QJQ^{-1}$ in the Jordan canonical form
$J=\mathrm{diag}[J_{1},\dots,J_{K}]$, where $J_{k}\in\R^{n_{k}}$ is the $k$-th Jordan block corresponding to the eigenvalue $\lambda_{k}$ of $Y$,
which may contain complex elements. The arguments below apply for the complex space $\mathbb C^{m}$ with the Euclidean norm $\|\cdot\|$.

Without loss of generality, we set $J_{1}=0$ corresponding to the single zero eigenvalue. Thus, the second term in (\ref{sol1}) can be transformed into
\begin{eqnarray*}
Q^{-1}\int_{0}^{t}\exp(Y(t-s))z(s)\,ds=\int_{0}^{t}\exp(J(t-s))\tilde{z}(s)\,ds
\end{eqnarray*}
with $\tilde{z}(s)=Q^{-1}z(s)$. The component corresponding to the Jordan block $J_{k}$ can be written as $\int_{0}^{t}\exp(J_{k}(t-s))z^{k}(s)\,ds$, where $z^{k}$ is the component vector corresponding to $J_{k}$.

We will show that $\int_{0}^{t}\exp(J_{k}(t-s))z^{k}(s)ds$ is bounded for each $k\ge 1$.
For each $k>1$, there exists a norm $\|\cdot\|_{k}$ such that
\begin{align*}
\left\|\int_{0}^{t}\exp(J_{k}(t-s))z^{k}(s)\,ds\right\|_{k}&\le\int_{0}^{t}
\left\|\exp(J_{k}(t-s))\right\|_{k}\|z^{k}(s)\|_{k}\,ds\\
&\le\int_{0}^{t}\exp(-\lambda_{k}(t-s))
\|z^{k}(s)\|_{k}\,ds
\end{align*}
because the eigenvalues of $\exp(J(t-s))$ are $\exp(\lambda_{k}(t-s))$. Since
${\mathcal Re}(\lambda_{k})<0$ and $z(s)$ ($\tilde{z}(s)$) is bounded, we conclude that $\int_{0}^{t}\exp(J_{k}(t-s))z^{k}(s)\,ds$ is bounded by some constant for all $k>1$.

Consider the component corresponding to $J_{1}=0$:
\begin{equation*}
\int_{0}^{t}\tilde{z}(s)ds=\sum_{q=0}^{\lfloor t/T\rfloor}\int_{qT}^{(q+1)T}\tilde{z}(s)ds+\int_{\lfloor t/T\rfloor T}^{t}\tilde{z}(s)ds
=\int_{\lfloor t/T\rfloor T}^{t}\tilde{z}(s)ds.
\end{equation*}
Since $\int_{t}^{t+T}z(t)=0$ for all $t\ge 0$, this term is bounded, because $\int_{qT}^{(q+1)T}\tilde{z}(s)ds=0$ for $t\ge 0$ and $z(s)$ ($\tilde{z}(s)$) is bounded. Hence $\int_{0}^{t}\exp(J(t-s))\tilde{z}(s)\,ds$ is bounded, and therefore one can see that $\Phi(t)$ is bounded. This proves the first statement of this proposition.

We next prove that the phase difference trajectories are asymptotically stable. (i) Under condition 1, namely that the graph associated with $\bar{L}$ has a spanning tree, a sufficiently small $\epsilon$ guarantees that the graphs of $L(t)$ have spanning trees for all $t\ge 0$. By Theorem \ref{thm1} we conclude that the PD trajectories of the time-varying system (\ref{TVKR}) under ${\mathbf H}_{2}$ are asymptotically stable.  (ii) Under condition 2, a sufficiently small $\epsilon$ guarantees that $\xi(L(t),r)<\xi(\bar{L},r)/2$ , which implies the PD trajectories are asymptotically stable by Theorem \ref{thm2}. (iii) Under condition 3, which is a special form of the arguments above since $J$ is diagonal, a sufficiently small $\epsilon$ guarantees that (\ref{alpha2}) holds for some $h>0$ and $\hat{\alpha}>0$. Hence, the PD trajectories are asymptotically stable by Corollary \ref{cor2}. This completes the proof.
\end{proof}

\begin{remark}
It can be seen that from the proof of Proposition \ref{propx} that, under the conditions of Proposition \ref{propx}, the phase difference trajectories are asymptotically periodic with period equal to that of the time-varying parameters, as a consequence of Proposition \ref{prop1}. The adiabatic case of a large $T$, the transition rate used in \cite{Pet} implies a slow (induced by the slow periodicity of the time-varying parameters) and small (induced by the small perturbation of the time-varying parameters) phase dynamics as well as the phase-difference trajectories.
\end{remark}

To illustrate with a numerical example, we generate a connected undirected Erd\H os-Renyi random graph with $m=20$ nodes
with linking probability $p=0.2$. Let $\bar{A}=[\bar{a}_{ij}]$ denote its adjacency matrix.
We set
\begin{eqnarray*}
\omega_{i}(t)=\bar{\omega_{i}}+\epsilon\sin(t+\alpha_{i}),\quad a_{ij}(t)=\begin{cases} 0, & \bar{a}_{ij}=0,\\
1+\epsilon\cos(t+\beta_{ij}), & \bar{a}_{ij}\ne 0,\end{cases}
\end{eqnarray*}
where the $\alpha_{i}$ and $\beta_{ij}$
are randomly picked in $[-r/2,r/2]$ with $r=\pi/3$, following a uniform distribution. We take $\epsilon=0.1$.
We simulate this system ten times with random initial values picked from the interval $[-r/2,r/2]$. For comparison, we also simulate the Kuramoto model (\ref{SKR}) with fixed frequencies and linking coefficients and the same initial values of those of (\ref{TVKR}). As shown in the top panel of Fig.~\ref{Fig1_small}, $\theta_{1}(t)$ is essentially indistinguishable from its first-order approximation
\begin{equation*}
\theta_{1}(t)\approx\bar{\theta}_{1}(t)+\epsilon\Theta_{1}(t)
\end{equation*}
with $\theta_{i}(0)=\bar{\theta}_{i}(0)$ and $\Theta_{i}(0)=0$ for all $i=\onetom$.
The PD trajectories of the time-varying Kuramoto network are asymptotically stable and the phase differences are close to those of the phase-locked difference of the static system (\ref{SKR}). In addition, $\mathscr{A}^{r}$ (with $r=\pi/3$) is indeed found to be invariant for (\ref{TVKR}).

\begin{figure}[htp]
\centering
\subfigure[Phase dynamics][]
{
\includegraphics[width=.8\textwidth]
{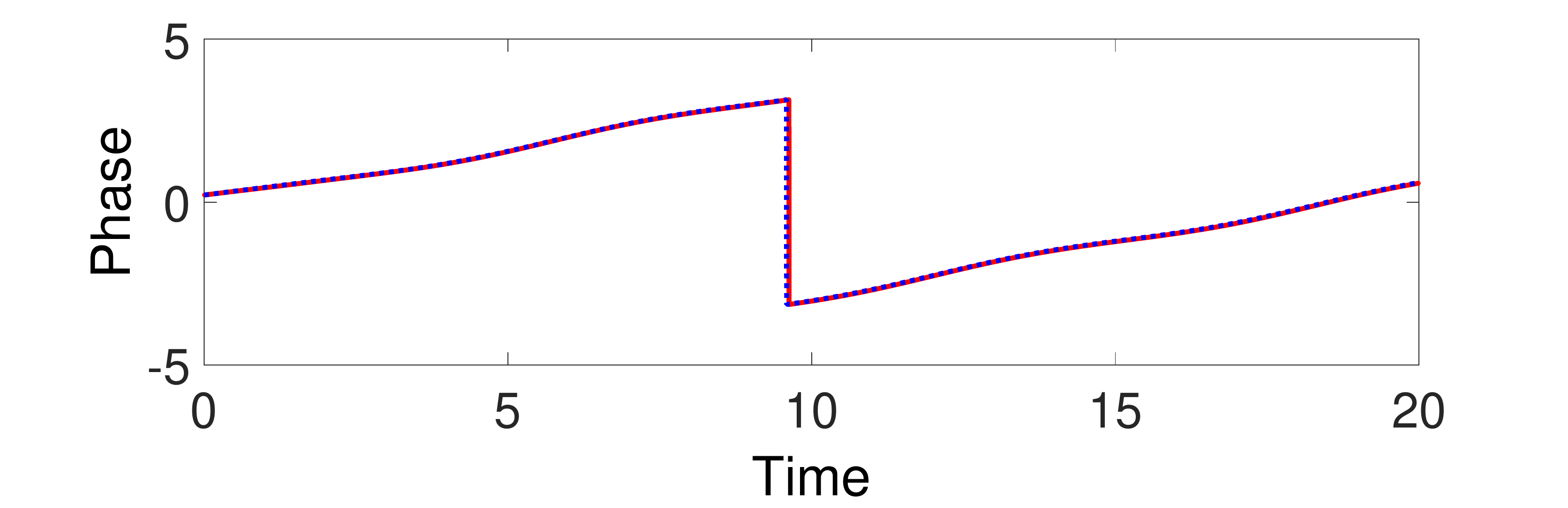}
}
\subfigure[Phase difference dynamics][]{
\includegraphics[width=.9\textwidth, height=4cm]
{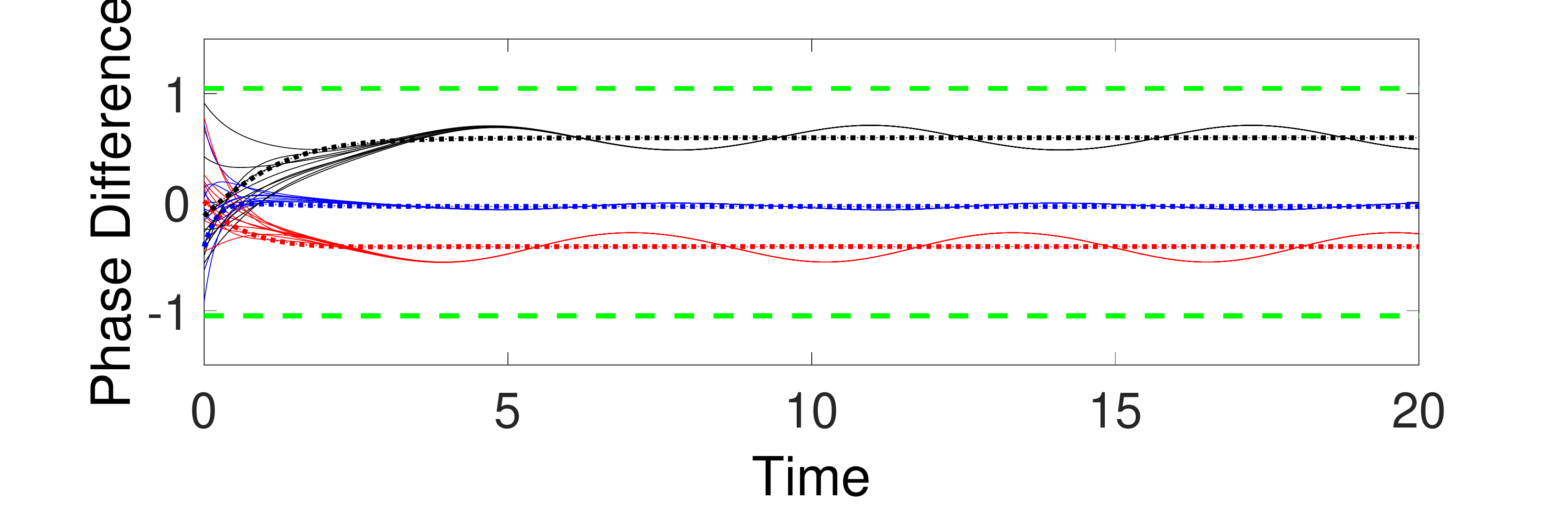}
}
\caption{Top panel: Dynamics of $\theta_{1}(t)$ (red solid line) and its approximation $\bar{\theta}_{1}(t)+\epsilon\Theta_{1}(t)$ (blue dashed line). Bottom panel: Dynamics of the phase differences $\theta_{1}(t)-\theta_{4}(t)$ (red solid lines), $\theta_{3}(t)-\theta_{7}(t)$ (blue solid lines), and $\theta_{17}(t)-\theta_{11}(t)$ (black solid lines) of ten simulations, and the comparisons: $\bar{\theta}_{1}(t)-\theta_{4}(t)$ (red dashed line), $\bar{\theta}_{3}(t)-\bar{\theta}_{7}(t)$ (blue dashed line), and $\bar{\theta}_{17}(t)-\bar{\theta}_{11}(t)$ (black dashed line). The two horizontal green dashed lines mark the values $\pm\pi/3$ corresponding to $\pm r$.}
\label{Fig1_small}
\end{figure}

\subsection{Fast Switching}

In this subsection, we consider the scenario that the time-variation of the parameters is due to fast switching near certain constants with speed $1/\epsilon$,
where $\epsilon > 0$ is a small parameter, analogously to \cite{Stil,Bely}.

Consider the following hypotheses.

${\bf H}_{3}$: $\omega_{i}(t)$ and $a_{ij}(t)$ are piecewise continuous, bounded, periodic functions with period $\epsilon T$ with average values
\begin{equation}  \label{avg}	
\bar{\omega}_{i} = \frac{1}{\epsilon T}\int_{0}^{\epsilon T}\omega_{i}(s)ds,
\qquad \bar{a}_{ij} = \frac{1}{\epsilon T}\int_{0}^{\epsilon T}a_{ij}(s)ds.
\end{equation}


By this hypothesis, let
\begin{equation}
\bar{l}_{ij}=\begin{cases}-\bar{a}_{ij},&i\ne j,\\
-\sum_{j=1}^{m}\bar{l}_{ij},&i=j,
\end{cases}\qquad \bar{L}:=[\bar{l}_{ij}],\label{barL}
\end{equation}

\begin{equation*}
\tilde{\omega}_{i}(s)=\omega_{i}(\epsilon s),\quad \tilde{a}_{ij}(s)=a_{ij}(\epsilon s)
\end{equation*}
and note that they are periodic functions with period $T$ and satisfy
\begin{equation*}
\frac{1}{T}\int_{t}^{t+T}\tilde{\omega}_{i}(s)ds=\bar{\omega}_{i},\quad
\frac{1}{T}\int_{t}^{t+T}\tilde{a}_{ij}(s)ds=\bar{a}_{ij}
\end{equation*}
for all $t$.

Let $\theta(t)=[\theta_{1}(t),\dots,\theta_{m}(t)]^{\top}$ be the solution of (\ref{TVKR}) with the time-varying parameters $\omega_{i}(t)$ and $a_{ij}(t)$ satisfying hypotheses ${\bf H}_{3}$, and $\bar{\theta}(t)=[\bar{\theta}_{1}(t),\dots,\bar{\theta}_{m}(t)]^{\top}$ be the solution of  (\ref{SKR}) with constant parameters $\bar{\omega}_{i}$ and $\bar{a}_{ij}$ as in \eqref{avg}. We assume that (\ref{SKR}) possesses a stable phase-locked equilibrium, denoted by $\bar{\theta}_{i}(t)$, with phase differences $\bar{\theta}_{ij}=\bar{\theta}_{i}-\bar{\theta}_{j}$ being constants in time.

Let $\Delta_{i}(t)=\theta_{i}(t)-\bar{\theta}_{i}(t)$, which obey
\begin{eqnarray}
\dot{\Delta}_{i}&=&[\tilde{\omega}_{i}(t/\epsilon)-\bar{\omega}_{i}]+\sum_{j=1}^{m}
[\tilde{a}_{ij}(t/\epsilon)-\bar{a}_{ij}]\sin(\bar{\theta}_{ji})\nonumber\\
&&
+\sum_{j=1}^{m}\tilde{a}_{ij}(t/\epsilon)[\sin(\theta_{ji}(t))-\sin(\bar{\theta}_{ji})]\nonumber\\
&=&[\tilde{\omega}_{i}(t/\epsilon)-\bar{\omega}_{i}]+\sum_{j=1}^{m}
[\tilde{a}_{ij}(t/\epsilon)-\bar{a}_{ij}]\sin(\bar{\theta}_{ji})\nonumber\\
&&
+\sum_{j=1}^{m}\tilde{a}_{ij}(t/\epsilon)\cos(\zeta_{ji}(t))
[\Delta_{j}-\Delta_{i}],\quad i=\onetom,
\label{variation2}
\end{eqnarray}
where $\zeta_{ji}\in[\min(\theta_{ji}(t),\bar{\theta}_{ji}),\max(\theta_{ji}(t),\bar{\theta}_{ji})]$ are picked by the mean-value theorem with $\cos(\zeta_{ij})=\cos(\zeta_{ij})$. We then have the following result.

\begin{proposition}  \label{prop2}
Let $r\in[0,\pi/2)$ and suppose that $\mathscr{A}^{r}$ is invariant for \eqref{TVKR},  ${\bf H}_{3}$ holds, \eqref{SKR} possesses a phase-locked solution $[\bar{\theta}_{1}(t),\dots,\bar{\theta}_{m}(t)]^{\top}\in\mathscr{A}^{r}$ as described by (\ref{ps}), and $L(t)$ is symmetric and positive semidefinite for all $t\ge 0$. If $\lambda_{2}(\bar{L})>0$, where $\bar{L}$ is defined in \eqref{barL}.
Then there exists some $\epsilon'>0$ such that the PD trajectories of \eqref{TVKR} have the form of $\theta_{ij}(t)=\bar{\theta}_{ij}+\epsilon\Upsilon_{ij}(t)$ as $t\to\infty$ for some functions $\Upsilon_{ij}(t)$ bounded with respect to $t>0$ and $\epsilon'>\epsilon>0$. In addition, if $\epsilon$ is sufficiently small, then the PD trajectories is asymptotically stable within $\mathscr A^{r}$.
\end{proposition}

Let
\begin{equation*}
r_{i}(s)=[\tilde{\omega}_{i}(s)-\bar{\omega}_{i}]+\sum_{j=1}^{m}
[\tilde{a}_{ij}(s)-\bar{a}_{ij}]\sin(\bar{\theta}_{ji}),
\end{equation*}
and $r(s)=[r_{1}(s),\dots,r_{m}(s)]^{\top}$,
which implies
$
\int_{s}^{s+T}r(\chi)d\chi=0
$
for all $s\ge 0$.  Let
\begin{equation*}
R_{ij}(t,\epsilon)=\begin{cases}\tilde{a}_{ij}(t/\epsilon)\cos(\zeta_{ji}(t))&i\ne j\\
-\sum_{k=1}^{m}R_{ik}(t,\epsilon)&i=j.\end{cases}
\end{equation*}
and define the matrix $R(t,\epsilon)=[R_{ij}(t,\epsilon)]_{i,j=1}^{m}$.
It can be seen that $R(t,\epsilon)$ is symmetric for all $t$ due to the symmetry of $L(t)$ and $\cos(\zeta_{ij})$.
Then (\ref{variation2}) can be rewritten in the compact form
\begin{equation}  \label{variation21}
	\dot{\Delta}=r(t/\epsilon)+R(t,\epsilon)\Delta(t),
	\quad \Delta(t)=[\Delta_{1}(t),\dots,\Delta_{m}(t)]^{\top}.
\end{equation}
We first prove a lemma as a preparation for the proof of Proposition \ref{prop2}.

\begin{lemma}\label{lem2}
Let $U(t,s;\epsilon)$ be the state-transition matrix of the linear system
\begin{equation}
\dot{z}=R(t,\epsilon)z(t) \label{linear2}
\end{equation}
and assume the conditions in Proposition \ref{prop2}. Then there exist  positive numbers $\epsilon'$, $M$, $T_{1}$, and $\alpha$ such that for each $s\ge 0$, the inequality
\begin{equation}
\|U(t,s;\epsilon)-\frac{1}{m}{\bf 1}\otimes{\bf 1}\|\le M\exp(-\alpha (t-s))\label{coverg}
\end{equation}
holds for all $\epsilon\in(0,\epsilon')$ and $t\ge s+T_{1}$. In addition, let
\begin{equation*}
\mathscr L=\{x=[x_{1},\dots,x_{m}]^{\top}:~\sum_{i=1}^{m}x_{i}=0\}.
\end{equation*}
Then $U(t,s;\epsilon) \mathscr L\subset\mathscr L$ and for each $s\ge 0$
\begin{equation*}
\|U(t,s;\epsilon)\|_{\mathscr L}\le M\exp(-\alpha (t-s))
\end{equation*}
for all $\epsilon\in(0,\epsilon')$ and $t\ge s+T_{1}$.
\end{lemma}
\begin{proof}
%
The conditions of Proposition \ref{prop2}, the symmetry of $R(t,\epsilon)$, and Theorem \ref{thm3} give
\begin{eqnarray*}
\frac{1}{T}\int_{t}^{t+T}R(s,\epsilon)\, ds\le \frac{\cos(r)}{2}\bar{L},
\end{eqnarray*}
which implies $\lambda_{2}[(1/T)\int_{t}^{t+T}R(s,\epsilon)ds]\le\cos(r)/2\lambda_{2}(\bar{L})$ is negative. Therefore, for each initial time $s$ and initial value $z(s)=z^{0}$, the solution of (\ref{linear2}), denoted by $U(t,s;\epsilon)z^{0}$, reaches consensus exponentially at a rate $O(\exp(-\alpha (t-s))$, for some $\alpha>0$ depending on $T$ and $\cos(r)/2\lambda_{2}(\bar{L})$.


Let $z(t)=U(t,s)z^{0}$. Noting the fact that
\begin{eqnarray*}
\frac{d}{dt}\left({\bf 1}^{\top}z(t)\right)={\bf 1}^{\top}R(t,\epsilon)z(t)=0
\end{eqnarray*}
for any $z^{0}\in\R^{m}$, we have ${\bf 1}^{\top}U(t,s;\epsilon)={\bf 1}$ by the symmetry of $R(t,\epsilon)$.  Hence, in both cases, from \cite{Chatt}, one can see that
\begin{align}
\lim_{t\to\infty}U(t,s;\epsilon)z^{0}=\zeta{\bf 1}\label{zeta2}
\end{align}
for some $\zeta\in\R$.

Since on the one hand
\begin{eqnarray*}
\frac{1}{m}{\bf 1}^{\top}U(t,s;\epsilon)z^{0}=\frac{1}{m}{\bf 1}^{\top}z^{0}
\end{eqnarray*}
and on the other hand
\begin{eqnarray*}
\frac{1}{m}{\bf 1}^{\top}{\bf 1}\zeta=\zeta,
\end{eqnarray*}
we have $\zeta=(1/m)\sum_{i=1}^{m}z^{0}_{i}$. Since this holds for all $z^{0}\in\R^{m}$, the first statement is proved.

For any $y=[y_{1},\dots,y_{m}]\in\mathscr L$, namely, $\sum_{i=1}^{m}y_{i}=0$, we have
\begin{eqnarray*}
{\bf 1}^{\top}U(t,s;\epsilon)y=U(t,s;\epsilon){\bf 1}^{\top}y=0,~\forall~t\ge s.
\end{eqnarray*}
In other words, $U(t,s)y\in\mathscr L$ for all $t\ge s$. Therefore, $U(t,s;\epsilon)\mathscr L\subset\mathscr L$.

Thus, by (\ref{coverg}) and the fact that ${\bf 1}^{\top}y=0$,  we have
\begin{eqnarray}
\|U(t,s;\epsilon)-\frac{1}{m}{\bf 1}{\bf 1}^{\top}y\|=\|U(t,s;\epsilon)y\|_{\mathscr L}\le M\exp(-\alpha(t-s))\|y\|,
\end{eqnarray}
which proves the second statement, and completes the proof of the lemma.
\end{proof}

Let $w(t,s;\epsilon)=\frac{\partial U(t,s;\epsilon)}{\partial s}$ and note that
\begin{align*}
\frac{\partial w(t,s;\epsilon)}{\partial t}&=\frac{\partial}{\partial t}\frac{\partial
U(t,s;\epsilon)}{\partial s}=\frac{\partial}{\partial s}\frac{\partial
U(t,s;\epsilon)}{\partial t}\\
&=\frac{\partial}{\partial s}R(t,\epsilon)U(t,s)=R(t,\epsilon)w(t,s),
\end{align*}
and $w(t,t;\epsilon)=R(t,\epsilon)$. Hence,
\begin{eqnarray*}
w(t,s;\epsilon)=U(t,s;\epsilon)R(s,\epsilon).
\end{eqnarray*}
This implies that each column vector of $w(t,s;\epsilon)$ is a bounded linear combination of the column vectors of $U(t,s;\epsilon)$. In addition,
\begin{align*}
{\bf 1}^{\top}w(t,s;\epsilon)={\bf 1}^{\top}U(t,s;\epsilon)R(s,\epsilon)={\bf 1}^{\top}R(s,\epsilon)=0,
\end{align*}
implying that $w(t,s;\epsilon)$ belongs to the subspace $\mathcal L$. Thus,
\begin{align}
\|w(t,s;\epsilon)\|\le M_{1}\exp(-\alpha(t-s))\label{w1}
\end{align}
for some $M_{1}>0$ and all $t>s\ge 0$.


\begin{proof}[Proof of Proposition \ref{prop2}] Let $\epsilon\in(0,\epsilon')$. We rewrite $U(t,s;\epsilon)$ and $w(t,s;\epsilon)$ as $U(t,s)$ and $w(t,s)$ respectively for simplicity.

The solution of (\ref{variation21}) has the form
\begin{eqnarray*}
\Delta(t)=U(t,0)\Delta(0)+\int_{0}^{t}U(t,\tau)r(\tau/\epsilon)d\tau.
\end{eqnarray*}
Equivalently,
\begin{align*}
\Delta(t)={\bf 1}\zeta+U(t,0)\Delta(0)-{\bf 1}\zeta+\epsilon O(t)
\end{align*}
with $O(t)=\frac{1}{\epsilon}\int_{0}^{t}U(t,\tau)r(\tau/\epsilon)d\tau$.
By Lemma \ref{lem2}, the term $U(t,0)\Delta(0)$ converges to ${\bf 1}\zeta$ for  $\zeta=\sum_{i=1}^{m}\Delta_{i}(0)$. That is, $\lim_{t\to\infty}U(t,0)\Delta(0)-{\bf 1}\zeta=0$.
The term $\epsilon O(t)$ becomes
\begin{eqnarray*}
\int_{0}^{t}U(t,\tau)r(\tau/\epsilon)d\tau=\sum_{n=0}^{K}\int_{n\epsilon T}^{(n+1)\epsilon T}U(t,\tau)r(\tau/\epsilon)d\tau
+\int_{K\epsilon T}^{t}U(t,\tau)r(\tau/\epsilon)\,d\tau,
\end{eqnarray*}
where $K=\lfloor(t-s)/(\epsilon T)\rfloor$. Using the fact that
\begin{eqnarray*}
U(t,\tau)=U(t,n\epsilon T)+(\tau-n\epsilon T)\int_{0}^{1}w(t,\lambda\tau+(1-\lambda)(n\epsilon T))\,d\lambda,
\end{eqnarray*}
we have
\begin{eqnarray*}
&&\int_{n\epsilon T}^{(n+1)\epsilon T}U(t,\tau)r(\tau/\epsilon)\, d\tau=
\int_{n\epsilon T}^{(n+1)\epsilon T}\bigg[U(t,n\epsilon T)\\
&&+(\tau-n\epsilon T)\int_{0}^{1}w(t,\lambda\tau+(1-\lambda)(n\epsilon T)\, d\lambda\bigg]~r(\tau/\epsilon)\, d\tau.
\end{eqnarray*}
Note
\begin{eqnarray*}
\int_{n\epsilon T}^{(n+1)\epsilon T}U(t,n\epsilon T)r(\tau/\epsilon)\, d\tau=
U(t,n\epsilon T)\epsilon\int_{nT}^{(n+1)T}r(\chi)\, d\chi=0.
\end{eqnarray*}
Using the fact  ${\bf 1}^{\top}R(t,\epsilon)=0$, the symmetry of $R(t,\epsilon)$, Lemma \ref{lem2}, and the inequality (\ref{w1}), we obtain
\begin{eqnarray*}
\|w(t,s)\|\le M_{2}\exp(-\alpha(t-s))\quad\forall~t>s\ge 0
\end{eqnarray*}
for some $M_{2}>0$ and $\alpha>0$. Thus, one can derive
\begin{eqnarray*}
&&\left\|\int_{n\epsilon T}^{(n+1)\epsilon T}(\tau-n\epsilon T)\int_{0}^{1}w(t,\lambda\tau+(1-\lambda)(n\epsilon T))\, d\lambda~r(\tau/\epsilon)\, d\tau\right\|\\
&\le&\epsilon M_{3}\exp(-\alpha(t-(n+1)\epsilon T))
\end{eqnarray*}
for some $M_{3}\ge M_{2}>0$. Hence,
\begin{eqnarray*}
&&\left\|\sum_{n=0}^{K}\int_{n\epsilon T}^{(n+1)\epsilon T}U(t,\tau)r(\tau/\epsilon)\, d\tau\right\|\\
&\le&\epsilon M_{3}\sum_{n=0}^{K}\exp(-\alpha(t-(n+1)\epsilon T))\le\epsilon M_{3}\frac{\exp(2\alpha\epsilon T)}{\exp(\alpha\epsilon T)-1}.
\end{eqnarray*}
In addition,
\begin{eqnarray*}
\left\|\int_{K\epsilon T}^{t}U(t,\tau)r(\tau/\epsilon)\, d\tau\right\|\le\epsilon M_{4}
\end{eqnarray*}
for some $M_{4}>0$ with $|U(t,\tau)r(\tau/\epsilon)|\le M_{4}$.

To sum up, noting that the constants $M_{1,2,3,4}$ are independent of $\epsilon$, one can conclude that the term $O(t)$ is bounded with respect to both $\epsilon$ and $t$.
Hence,
\begin{eqnarray*}
\Delta(t)\sim {\bf 1}\zeta+\epsilon O(t),\quad \text{as } t\to\infty
\end{eqnarray*}


Let $\Upsilon_{ij}(t)=O_{j}(t)-O_{i}(t)$, which are  bounded with respect to $t\ge 0$ and $\epsilon>0$, and $\bar{\theta}_{ij}$ are the PDs of the phase-locked equilibrium of (\ref{TVKR}) when $\omega_{i}(t)\equiv \bar{\omega}_{ij}$ and $a_{ij}(t)\equiv \bar{a}_{ij}$. Thus,
the PD of (\ref{TVKR}) can be written in the form
\begin{align*}
\theta_{ij}(t)&=\bar{\theta}_{i}(t)-\bar{\theta}_{j}(t)+\Delta_{i}(t)-\Delta_{j}(t)\\
&\sim\bar{\theta}_{ij}+\epsilon\Upsilon_{ij}(t),~{\rm as}~t\to\infty.
\end{align*}
This completes the proof.
\end{proof}


To illustrate, we consider a network of five Kuramoto oscillators whose coupling matrix switches between the following two symmetric matrices:
\begin{eqnarray*}
L^{1}&=&\left[\begin{array}{lllll}
-1.6793  & -0.3012  &  2.3645 &  -0.2241 &  -0.1599\\
   -0.3012 &  -1.0878  &  1.0473 &  -0.4689  &  0.8106\\
    2.3645  &  1.0473  & -3.3379 &  -0.4142  &  0.3403\\
   -0.2241 &  -0.4689  & -0.4142 &  -0.4065  &  1.5137\\
   -0.1599  &  0.8106 &   0.3403  &  1.5137 &  -2.5046
\end{array}\right]\\
L^{2}&=&\left[\begin{array}{lllll}
-8.4835&    1.6123&    2.5756  &  2.1175  &  2.1780\\
    1.6123&   -4.3012&    2.2760&    0.5141&   -0.1013\\
    2.5756 &   2.2760 &  -8.3439 &   2.1106 &   1.3817\\
    2.1175  &  0.5141  &  2.1106  & -4.6359  & -0.1064\\
    2.1780   &-0.1013   & 1.3817   &-0.1064   &-3.3521\\
\end{array}\right]
\end{eqnarray*}
and the intrinsic frequency vector switches between the following two vectors:
\begin{eqnarray*}
\omega^{1}&=&[1.3468,    0.0850,    1.8434,    1.9853,    1.1750]^{\top},\\
\omega^{2}&=&[2.2854,    0.6908,    2.4129,    0.5544,    2.7517]^{\top}.
\end{eqnarray*}
(The parameters of this example $L^{1,2}$ and $\omega^{1,2}$ are randomly generated until the specific criteria of Proposition \ref{prop2} are met.)
The system is switched with a frequency $h$. It can be checked that $\Theta_{r}$ with $r=\pi/3$ is invariant for the switched system, and $\lambda_{2}((L^{1}+L^{2})/2)=-2.5004$. Therefore, by Proposition \ref{prop2}, the PD trajectories asymptotically approach those of the averaged system as $h\to \infty$.
As shown in Fig.~\ref{Fig1_fast}, the averaged system of Kuramoto model possesses a phase-locked equilibrium. As the switching frequency increases from $10$ Hz to $50$ Hz, the PD dynamics asymptotically converge to the phase-locked equilibrium as $t\to\infty$, provided $\epsilon$ is sufficiently small (i.e., the switching frequency is sufficiently high).
\begin{figure}[htp]
\centering
\subfigure[Phase dynamics][]
{
\includegraphics[width=.9\textwidth]
{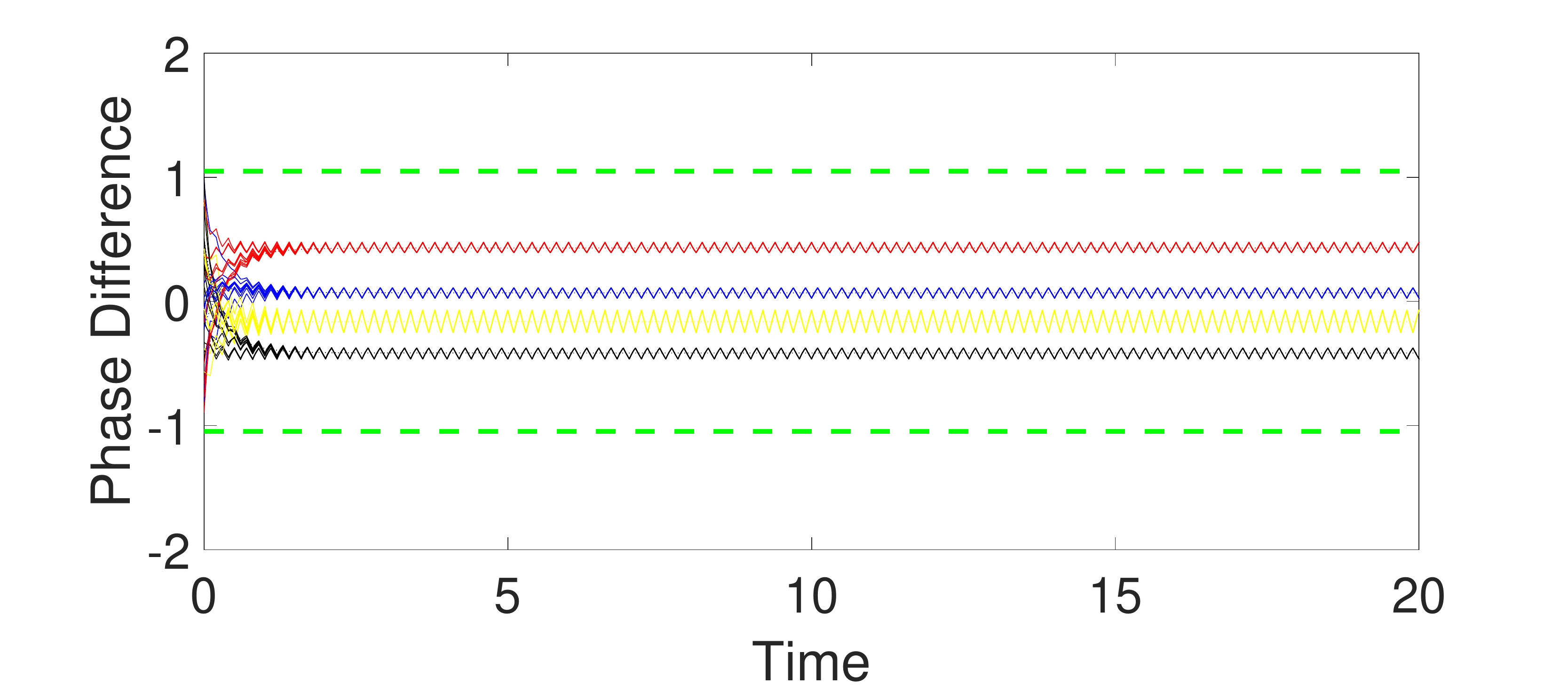}
}
\subfigure[Phase difference dynamics][]{
\includegraphics[width=.9\textwidth]
{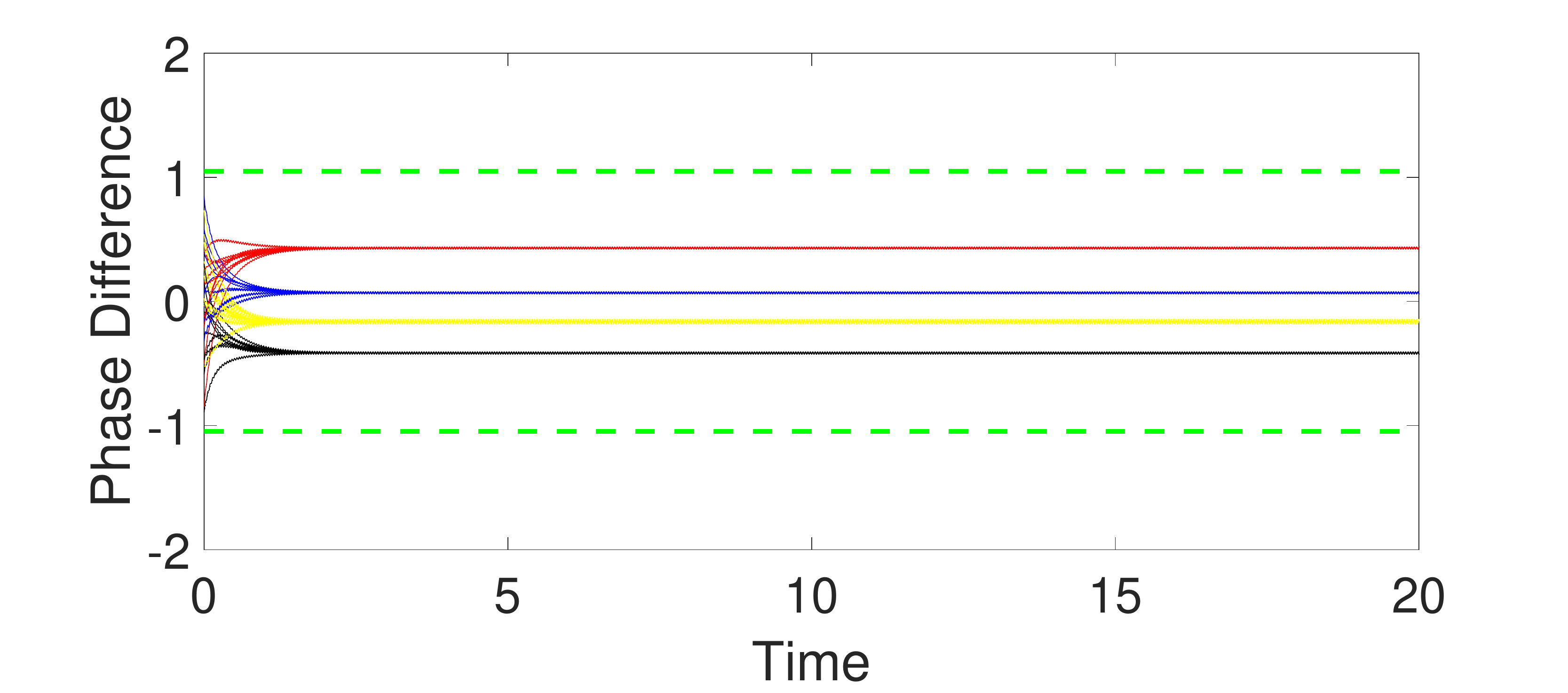}
}
\caption{Evolution of the phase differences of the switched Kuramoto model, $\theta_{1}(t)-\theta_{2}(t)$ (red), $\theta_{2}(t)-\theta_{3}(t)$ (black), $\theta_{3}(t)-\theta_{4}(t)$ (blue), and $\theta_{4}(t)-\theta_{5}(t)$ (yellow), asymptotically approaching constant values in ten simulations starting from randomly chosen initial values in $[-\pi/3,\pi/3]$.
The two horizontal green dashed lines mark the values $\pm\pi/3$ corresponding to $\pm r$.  The switching frequency $h$ is 10 Hz in the top Panel and 50 Hz in the bottom panel.}
\label{Fig1_fast}
\end{figure}

\section{Conclusion}
When the couplings and intrinsic frequencies vary in time, the Kuramoto model cannot maintain phase-locking states {\color{blue} when the number of oscillators is finite}. In this paper, we have studied asymptotical stability of non-equilibrium phase-unlocking dynamics. Assuming that the PDs remain in the interval $[-\pi/2,\pi/2]$ whenever the initial differences do, we have derived sufficient conditions for the asymptotical stability of PDs. As a particular novelty, we have allowed negative couplings in the analysis.
Moreover, we have identified and proved asymptotic PD dynamics in various scenarios and illustrated them by numerical examples. In a future investigation, we will study the situation when the phase differences may be larger than $\pi/2$ and the couplings and intrinsic frequencies may be stochastically changing.

\section*{Acknowledgement}
The authors thank the anonymous reviewers for their constructive comments that
helped improve the paper significantly. W. L. Lu is jointly supported by the National Natural Sciences Foundation
of China under Grant No.~61673119, the Key Program of the National Science Foundation of China No.~91630314, the Laboratory of Mathematics for
Nonlinear Science, Fudan University, and the Shanghai Key Laboratory for Contemporary Applied Mathematics, Fudan University. The authors gratefully acknowledge the support of the ZiF, the Center for Interdisciplinary Research
of Bielefeld University, where part of this research was conducted
under the cooperation program \emph{Discrete and Continuous Models in the Theory of
Networks}.

\section*{Appendix A}
\begin{proof}[Proof of Lemma \ref{lem0}]
Let $t^{*}=\sup\{t:\theta(\tau)\in\mathscr{A}^{r}~\forall~\tau\in[0,t)\}$. We shall prove Lemma \ref{lem0} by showing $t^{*}=\infty$. Assume not. Then for each index $i^{*}$ with $\theta_{i^{*}}(t^{*})=\max_{i}\theta_{i}(t^{*})$ and each $j_{*}$ with $\theta_{j_{*}}(t^{*})=\min_{j}\theta_{j}(t^{*})$, we have $\theta_{i^{*}}(t^{*})-\theta_{j_{*}}(t^{*})=r$. Note that
\begin{align*}
&a_{i^{*}j}(t^{*})\sin(\theta_{j}(t^{*})-\theta_{i^{*}}(t^{*}))
\le-\sin(r)[a_{i^{*}j}(t^{*})]^{-}\\
&a_{j_{*}k}(t^{*})\sin(\theta_{k}(t^{*})-\theta_{j_{*}}(t^{*}))
\ge\sin(r)[a_{j_{*}k}(t^{*})]^{-}
\end{align*}
and when $j\in\Lambda_{i^{*}j_{*}}(t^{*})$ (i.e., $a_{i^{*}j}(t^{*})>0$ and
$a_{j_{*}j}(t^{*})>0$),
\begin{align*}
&a_{i^{*}j}(t^{*})\sin(\theta_{j}(t^{*})-\theta_{i^{*}}(t^{*}))-
a_{j_{*}j}(t^{*})\sin(\theta_{j}(t^{*})-\theta_{j_{*}}(t^{*}))\\
&\le-\min\{a_{i^{*}j}(t^{*}),a_{j_{*}j}(t^{*})\}\left[
\sin(\theta_{i^{*}}(t^{*})-\theta_{j}(t^{*}))
+\sin(\theta_{j}(t^{*})-\theta_{j_{*}}(t^{*}))\right].\\
&\le-\min\{a_{i^{*}j}(t^{*}),a_{j_{*}j}(t^{*})\}\sin(r).
\end{align*}
Therefore,
\begin{align*}
&\dot{\theta}_{i^{*}}-\dot{\theta}_{j_{*}}\left|_{t=t^{*}}\right.
=\omega_{i^{*}}(t^{*})-\omega_{j_{*}}(t^{*})-[a_{i^{*}j_{*}}(t^{*})
+a_{j_{*}i^{*}}(t^{*})]\sin(r)\\
&+\sum_{j\ne j_{*}}a_{i^{*}j}(t^{*})
\sin(\theta_{j}(t^{*})-\theta_{i^{*}}(t^{*}))-\sum_{k\ne i^{*}}a_{j_{*}k}(t^{*})
\sin(\theta_{k}(t^{*})-\theta_{j_{*}}(t^{*}))\\
&\le \omega_{i^{*}}(t)-\omega_{j_{*}}(t^{*})-[a_{i^{*}j_{*}}(t^{*})
+a_{j_{*}i^{*}}(t^{*})]\sin(r)\\
&-\sum_{j\notin \Lambda_{i_{*}j^{*}}(t^{*}),j\ne i^{*},j_{*}}\{[a_{i^{*}j}(t^{*})]^{-}+[a_{i^{*}j}(t^{*})]^{-}\}\sin(r)\\
&-\sum_{k\in \Lambda_{i_{*}j^{*}}(t^{*})}\min\{a_{i^{*}j}(t^{*}),a_{j_{*}j}(t^{*})\}\sin(r)< 0.
\end{align*}
Thus $\theta_{i^{*}}(t)-\theta_{j_{*}}(t)$, and hence $\max_{i}\theta_{i}(t)-\min_{i}\theta_{i}(t)$, decreases in a small time interval starting at $t=t^{*}$. This contradicts the definition of $t^{*}$. Therefore, $t^{*}=\infty$.
\end{proof}

\section*{Appendix B}

\begin{proof}[Proof of Lemma \ref{lem1}]
Since $L$ is symmetric, $\widetilde{L}^{r}$ is symmetric with all row sums equal to $0$. Hence, $\widetilde{L}^{r}-L$ is a symmetric Metzler matrix with all row sums equal to zero, and is negative semidefinite because it is semi-diagonally dominant; so, all its eigenvalues are non-positive. Thus, for each $x\in\R^{n}$ with $x^{\top}{\bf 1}=0$, we have
\begin{eqnarray*}
x^{\top}\widetilde{L}^{r}x\le x^{\top}Lx.
\end{eqnarray*}
Therefore, $\chi_{1}\ge\chi_{2}$.
\end{proof}

\section*{Appendix C}
\begin{proof}[Proof of Lemma \ref{lem2}]
The idea of the proof of this lemma comes from \cite{Guo1994} with necessary modifications, in particular towards continuous-time systems.

From the hypotheses on $G(t)$, one can see that $\lambda_{1}(G(t))=0$. Let $P$ be an arbitrary orthogonal matrix whose first column equals ${\bf 1}/\sqrt{m}$. Since $G(t){\bf 1}=0$ for all $t$, we can write
\begin{eqnarray*}
P^{\top}G(t)P=\left[\begin{array}{cc}0&0\\
0&C(t)\end{array}\right]
\end{eqnarray*}
for some symmetric and positive semidefinite $C(t)\in\R^{m-1,m-1}$. Furthermore,
$\lambda_{2}(G(t))=\lambda_{1}(C(t))$.
Let $y=P^{\top}x$, $y=[y_{1},z]^{\top}$ with $y_{1}\in\R$. By (\ref{ma}),
\begin{eqnarray*}
\begin{cases}\dot{y}_{1}=0\\
\dot{z}=-C(t)z\end{cases}.
\end{eqnarray*}
Consider the linear time-varying system
\begin{eqnarray}
\dot{z}=-C(t)z\label{pr}
\end{eqnarray}
and let $U(t,s)$ be its state-transition matrix for $t\ge s$.
We shall show that
\begin{eqnarray}
\lambda_{m-1}\left[U^{\top}((k+1)h,kh)U((k+1)h,kh)\right]\le 1-\frac{h\beta_{k}}{(1+Rh)^{2}}.
\label{spec_est}
\end{eqnarray}
%
To this end, let $z^{k}$ be the unit eigenvector of $U^{\top}((k+1)h,kh)U((k+1)h,kh)$ associated with its largest eigenvalue, denoted by $\rho_{k}$. Thus, letting $z^{k+1}=U((k+1)h,kh)z^{k}$, which is a solution of (\ref{pr}) with $z(kh)=z^{k}$, denoted by $z(s)$ at $s=(k+1)h$, we have
\begin{eqnarray*}
\|z^{k+1}\|^{2}={z^{k}}^{\top}U^{\top}((k+1)h,kh)U((k+1)h,kh)z^{k}
=\rho_{k}.
\end{eqnarray*}
Noting that
\begin{eqnarray*}
z^{k+1}=z^{k}+\int_{kh}^{(k+1)h}[-C(s)]z(s)ds,
\end{eqnarray*}
and that $C(t)$ is positive semidefinite,
we have
\begin{eqnarray}
&&\|z(t)-z^{k}\|^{2}=\left\|\int_{kh}^{t}[-C(s)]z(s)\, ds\right\|^{2}\nonumber\\
&&\le\left\{\int_{kh}^{t}\|[C(s)]^{1/2}z(s)\|^{2}\, ds\right\}
\left\{\int_{kh}^{t}\|[C(s)]^{1/2}z(s)\|^{2}\, ds\right\}\nonumber\\
&&\le Rh\int_{kh}^{(k+1)h}z(s)^{\top}C(s)z(s)\, ds\label{A1}
\end{eqnarray}
for all $t\in[kh,(k+1)h]$. From the definition of $\beta_{k}$, we have
\begin{eqnarray*}
&&\beta_{k}^{1/2}\sqrt{h}\le \left\{{z^{k}}^{\top}\int_{kh}^{(k+1)h}[C(s)]\,ds \,z^{k}\right\}^{1/2}
=\left\{\int_{kh}^{(k+1)h}\|[C(s)]^{1/2}z^{k}\|^{2}\,ds \right\}^{1/2}\\
&&\le \left\{\int_{kh}^{(k+1)h}\|[C(s)]^{1/2}z(s)\|^{2}\,ds \right\}^{1/2}+\left\{\int_{kh}^{(k+1)h}\|[C(s)]^{1/2}\|^{2}\|z^{k}-z(s)\|^{2}\, ds \right\}^{1/2}\\
&&\le \left\{\int_{kh}^{(k+1)h}z^{\top}(s)[C(s)]z(s)\,ds \right\}^{1/2}+\sqrt{R}\left\{\int_{kh}^{(k+1)h}\|z^{k}-z(s)\|^{2}\, ds \right\}^{1/2}
\end{eqnarray*}
which, combined with (\ref{A1}), implies that
\begin{eqnarray*}
\beta_{k}^{1/2}\sqrt{h}\le(1+Rh)\left\{\int_{kh}^{(k+1)h}z(s)^{\top}[C(s)]z(s)\, ds \right\}^{1/2},
\end{eqnarray*}
that is,
\begin{eqnarray*}
\int_{kh}^{(k+1)h}z(s)^{\top}[C(s)]z(s)\, ds\ge \frac{\beta_{k}h}{(1+Rh)^{2}}.
\end{eqnarray*}
Note that
\begin{eqnarray*}
\frac{d}{dt}z^{\top}(t)z(t)=-2z(t)^{\top}C(t)z(t),
\end{eqnarray*}
which implies
\begin{eqnarray*}
\rho_{k} = {z^{k+1}}^{\top}z^{k+1}=1-2\int_{kh}^{(k+1)h}z(s)C(s)z(s)\, ds
\le 1-\frac{2\beta_{k}h}{(1+Rh)^{2}}.
\end{eqnarray*}
This proves (\ref{spec_est}), and yields $h\beta_{k}/[(1+Rh)^{2}]<1$.
Therefore,
\begin{eqnarray}
&&\|z(nh)\|^{2}=\|U(nh,(n-1)h)x((n-1)h)\|^{2}\le \left[1-\frac{h\beta_{n}}{(1+Rh)^{2}} \right]\|z((n-1)h)\|^{2}\nonumber\\
&&\le\prod_{k=0}^{n}\left[1-\frac{h\beta_{k}}{(1+Rh)^{2}} \right]\|z(0)\|^{2}.  \label{haha}
\end{eqnarray}
Since $\sum_{k=0}^{\infty}\beta_{k}=+\infty$, we conclude $\lim\limits_{n\to\infty}\|z(nh)\|=0$. Moreover, for $t\ge 0$ and $p:=\lfloor t/h\rfloor$,
\begin{eqnarray*}
\|z(t)\|\le\exp(R(t-ph))\|z(ph)\|\le\exp(Rh)\|z(ph)\|
\end{eqnarray*}
since $\|C(t)\|\le R$ for all $t\ge 0$. Thus, $\lim_{t\to\infty}\|z(t)\|=0$. In other words, $\lim_{t\to\infty}y(t)=[y(0),0,\dots,0]^{\top}$. Using the definition of $P$, we conclude
\begin{eqnarray*}
\lim\limits_{t\to\infty}x(t)=\lim\limits_{t\to\infty}Py(t)=y(0){\bf 1},
\end{eqnarray*}
that is, the system reaches consensus.
Furthermore, if $\beta_{k}>\beta_{0}$ for all $k$, it can be seen from (\ref{haha}) that
\begin{eqnarray*}
\|z(t)\|\le\exp(Rh) \, \|z(ph)\|\le\exp(Rh) \, \gamma^{p}\|z(0)\|,
\end{eqnarray*}
where $\gamma=\left[1-\frac{h\beta_{0}}{(1+Rh)^{2}}\right]$. Hence the convergence is exponential.
\end{proof}

\section*{Appendix D}
\begin{proof}[Proof of Lemma \ref{lem3}]
This claim trivially holds for conditions 1 and 3 in Proposition \ref{propx}.
In fact, under condition 2, assume that $Z$ has some eigenvalues with positive real parts, which implies that the linear system
\begin{eqnarray}
\dot{u}=Zu\label{linear1}
\end{eqnarray}
is unstable and unbounded for almost every initial condition. Here $u=[u_{1},\dots,u_{m}]^{\top}$. However, by similar arguments as in the proof of Theorem \ref{thm2}, we can conclude that (\ref{linear1}) reaches consensus, namely, $\lim_{t\to\infty}(u_{i}(t)-u_{j}(t))=0$ for all $i,j$. This implies that for any set of initial values there exists some $u_{0}$ such that $\lim_{t\to\infty}u_{i}(t)=u_{0}$ for all $i$.
This contradicts the assumption of eigenvalues having positive real parts, and completes the proof of the claim.
\end{proof}

\end{document}